\documentclass{daj}
\usepackage{amsmath, amsthm, amsopn, amssymb}
\usepackage{mathrsfs} 
\usepackage{mathscinet}
\usepackage{bbm}
\usepackage{enumerate, tikz, etoolbox, intcalc, geometry, caption, subcaption, afterpage}
\usepackage{enumitem}
\usepackage{float}

\usepackage[nameinlink]{cleveref}
\crefname{theorem}{Theorem}{Theorems}
\crefname{thm}{Theorem}{Theorems}
\crefname{mainthm}{Theorem}{Theorems}
\crefname{conj}{Conjecture}{Theorems}
\crefname{lemma}{Lemma}{Lemmas}
\crefname{lem}{Lemma}{Lemmas}
\crefname{remark}{Remark}{Remarks}
\crefname{prop}{Proposition}{Propositions}
\crefname{defn}{Definition}{Definitions}
\crefname{corollary}{Corollary}{Corollaries}
\crefname{cor}{Corollary}{Corollaries}
\crefname{section}{Section}{Sections}
\crefname{figure}{Figure}{Figures}
\crefname{quest}{Question}{Questions}

\newcommand{\N}{\mathbb{N}}
\newcommand{\Z}{\mathbb{Z}}
\newcommand{\Q}{\mathbb{Q}}
\newcommand{\R}{\mathbb{R}}
\newcommand{\C}{\mathbb{C}}
\newcommand{\FF}{\vec{F}}
\newcommand{\ind}{\mathbf{1}}

\newcommand{\Tor}{\operatorname{Tor}}
\newcommand{\stab}{\operatorname{stab}}
\newcommand{\rank}{\operatorname{rank}}
\newcommand{\spn}{\operatorname{span}}

\newcommand{\SI}[1]{{#1}-hyperplane repetition property}
\newcommand{\tF}{F^{\Tor}}

\newcommand{\absolute}[1] {\left|{#1}\right|}
\newcommand{\ignore}[1]{}

\graphicspath{ {./Pictures/} }

\font\si = cmssi12 scaled \magstep0

\long\def\comyaar#1{\ifdraft{\color{red}\si Yaar says ``#1'' }\else\ignorespaces\fi}

\newtheorem{thm}{Theorem}[section]
\newtheorem{lemma}[thm]{Lemma}
\newtheorem{prop}[thm]{Proposition}
\newtheorem{cor}[thm]{Corollary}

\theoremstyle{definition}
\newtheorem{definition}[thm]{Definition}
\newtheorem{example}[thm]{Example}
\newtheorem{remark}[thm]{Remark}

\dajAUTHORdetails{%
  title = {Periodicity of joint co-tiles in $\Z^d$}, 
  author = {Tom Meyerovitch, Shrey Sanadhya, and Yaar Solomon},
  plaintextauthor = {Tom Meyerovitch, Shrey Sanadhya, and Yaar Solomon},
  plaintexttitle = {Periodicity of joint co-tiles}, 
  runningtitle = {Periodicity of joint co-tiles}, 
  runningauthor = {Tom Meyerovitch, Shrey Sanadhya, and Yaar Solomon},
  copyrightauthor = {Tom Meyerovitch, Shrey Sanadhya, and Yaar Solomon},
  keywords = {translational tilings, periodic tiling conjecture, tiling equations.},
}   
\dajEDITORdetails{%
   year={2024},
   number={13},
   received={14 November 2023},   
   published={14 November 2024},  
   doi={10.19086/da.125855},       
}   

\begin{document}

\begin{frontmatter}[classification=text]
\author[tm]{Tom Meyerovitch\thanks{Supported by Israel Science Foundation grant no. 1052/18 and 985/23}}
\author[sa]{Shrey Sanadhya\thanks{Supported by Israel Science Foundation grant no. 1052/18}}
\author[ys]{Yaar Solomon}

\begin{abstract}
An old theorem of Newman asserts that any tiling of $\mathbb{Z}$ by a finite set is periodic. A few years ago, Bhattacharya proved the periodic tiling conjecture in $\Z^2$. Namely, he proved that for a finite subset $F$ of $\Z^2$, if there exists $A \subseteq \Z^2$ such that $F \oplus A = \Z^2$ then there exists a periodic $A' \subseteq \Z^2$ such that $F \oplus A' = \Z^2$.  The recent refutation of the periodic tiling conjecture in high dimensions due to Greenfeld and Tao motivates finding different generalizations of Newman's theorem and of Bhattacharya's theorem that hold in arbitrary dimension $d$. In this paper, we formulate and prove such generalizations. We do so by studying the structure of joint co-tiles in $\mathbb{Z}^d$. Our generalization of Newman's theorem states that for any $d \ge 1$, any joint co-tile for $d$ independent tiles is periodic.  For a $(d-1)$-tuple of finite subsets of $\Z^d$ that satisfy a certain technical condition that we call property $(\star)$, we prove that any joint co-tile decomposes into disjoint $(d-1)$-periodic sets. Consequently, we show that for a $(d-1)$-tuple of finite subsets of $\Z^d$ that satisfy property $(\star)$, the existence of a joint co-tile implies the existence of periodic joint co-tile.  Conversely, we prove that if a finite subset $F$ in $\mathbb{Z}^d$ admits a periodic co-tile $A$, then there exist $(d-1)$ additional tiles that together with $F$ are independent and admit $A$ as a joint co-tile, so that the first $(d-2)$ of these tiles together with $F$ satisfy property $(\star)$. Combined, our results give a new necessary and sufficient condition for a subset of $\mathbb{Z}^d$ to tile periodically. We also discuss tilings and joint tilings in other countable abelian groups.
\end{abstract}

\end{frontmatter}

\section{Introduction} 
Let $\Gamma$ be a countable abelian group. 
Given $A,B \subseteq \Gamma$ the \emph{sumset} $A+B$ is given by
\[
A+B = \left\{ a+b ~:~ a \in A , ~ b\in B \right\}.
\]
We write $A\oplus B = C$ if every $c \in C$  has a unique representation as $c = a+ B $ with $a \in A$ and $b \in B$.
We denote the cardinality of a set $F$ by $|F|$.
We write $F \Subset \Gamma$ to indicate that $F$ is a finite subset of $\Gamma$.

One says that $F$ \emph{tiles} $\Gamma$ if there exists a
 collection of disjoint union of translates of $F$ whose union is equal to $\Gamma$. That is, $F$ tiles $\Gamma$ if
there exists a set $A \subseteq \Gamma$ such that 
\begin{equation}\label{eq:tiling}
    F\oplus A = \Gamma.
\end{equation} 
 In that case, we say that $A$ is a  \textit{co-tile} for the \textit{tile} $F$. Let $g,h: \Gamma \rightarrow \R$, where $\Gamma$ is a countable abelian group. When at least one of the functions $g,h: \Gamma \rightarrow \R$ is finitely supported, the  \emph{convolution} $g * h:\Gamma \rightarrow \R$  is well defined and  given by 
\[
g*h (x) = \sum_{y \in \Gamma} g(y) \cdot h(x-y).
\]

Using this notation, equation \eqref{eq:tiling} is equivalent to $\ind_F * \ind_A = 1$, where $\ind_X$ denotes the indicator function of the set $X$.

Elements  $g_1,\ldots,g_k \in \Gamma$ of the abelian group $\Gamma$ are called \emph{independent} if the only integers $n_1,\ldots,n_k \in \Z$ that satisfy  $\sum_{j=1}^k n_j g_j =0$ are $n_1=\ldots = n_k=0$.
The \emph{rank} of an abelian group is the maximal size of an independent set.

 Suppose that $\Gamma$ is an abelian group of rank $d$ and that $k\le d$. A set $C \subseteq \Gamma$ is called \textit{$k$-periodic} if there exists a subgroup $L \le \Gamma$, with $\rank(L)\ge k$, such that $C + L = C$. In the case that $k=d$ we will also say that $C$ is \emph{periodic} instead of $d$-periodic.  
 We say that a tile set $F \Subset \Gamma$ \textit{tiles $\Gamma$ periodically} if there exits a periodic co-tile for $F$. If $F$ tiles $\Gamma$ but does not admit a periodic co-tile, then the set $F$ is called \emph{aperiodic}. 

Newman  \cite{Newman_1977} proved that any tiling of $\Gamma=\Z$ by a finite set is periodic. 
Already for $\Gamma = \Z^2$, it is not difficult to find tilings of $\Gamma$ by a finite set that are not even $1$-periodic. See  \cite[\S 1.3]{Greenfeld_Tao1_2021} for some examples and a brief discussion.
Still, it is natural to ask for different generalizations of Newman's theorem to higher-rank abelian groups. It has been conjectured for some time that for any $F \Subset \Z^d$, if there exists $A \subseteq \Z^d$ such that $F \oplus A = \Z^d$ then there exists a periodic $A' \subseteq \Z^d$ such that $F \oplus A' = \Z^d$ \cite{Lagarias_Wang_1996}, \cite{Grunbaum_Shephard_1987}. This conjecture became known as the \emph{periodic tiling conjecture}. The periodic tiling conjecture can be interpreted as an attempt to generalize Newman's theorem. The $\Z^2$ case of the periodic tiling conjecture was proved several years ago by Bhattacharya \cite{BPeriodicity2020}.
Other instances of the periodic tiling conjecture have been proved, under additional assumptions  \cite{Beauquier_Nivat_1991,Khetan_2021,Kenyon_1992,Szegedy_1998,Wijshoff_Leeuwen_1984}. The periodic tiling conjecture has recently been disproved for sufficiently large $d$ by Greenfeld and Tao~\cite{Greenfeld_Tao_2022}.

 The recent refutation of the periodic tiling conjecture in high dimensions motivates finding different generalizations of Newman's theorem and of Bhattacharya's theorem that hold in arbitrary dimension $d$. In this paper, we formulate and prove such generalizations.  Our approach is to study the structure of sets $A \subseteq \Z^d$  that satisfy 
\begin{equation}\label{eq:system_of_tiling_equations}
    F_j \oplus A = \Z^d \mbox{ for all } j=1,\ldots,k,
\end{equation}
for  subsets $F_1,\ldots,F_k \Subset \Z^d$.
We refer to such an $A$ as a \emph{joint co-tile} for $F_1,\ldots,F_k$.

In \cite{Greenfeld_Tao2_2021}, sets $A \subseteq \Z^d$ satisfying \eqref{eq:system_of_tiling_equations} have been referred to as solutions to the \emph{system of tiling equations}. It is not difficult to see that the existence of a joint co-tile for $F_1,\ldots,F_k \Subset \Z^{d}$ implies that $|F_1| = |F_2|=\ldots = |F_k|$ (see \Cref{prop:mean}).

Observe that whenever $A \subseteq \Z^d$ is a co-tile for $F \Subset \Z^d$, for every $k$ there exist  tiles $F_1,\ldots,F_k \Subset \Z^d$ admitting $A$ as a joint co-tile. For instance, for every
 $v_1,\ldots,v_k \in \Z^d$,  the set $A$ is a co-tile for $F_1,\ldots,F_k \Subset \Z^d$, where $F_i = F + v_i$, $i = 1,\ldots,k$. Furthermore, by the Dilation lemma \cite[Lemma 3.1]{Greenfeld_Tao_2022} (see also \Cref{lem:dilation} below) for any $F\Subset \Z^d$, and any $k \in \N$ there exist $r_1,\ldots,r_k \in \N$, such that any co-tile of $F$ is also a co-tile for $F_1,\ldots,F_k$, where $F_i := r_iF = \{ r_i f~:~f \in F\}$. Thus, in general, joint co-tiles for $k$-tuples of tiles $F_1,\ldots,F_k$ have no more structure than co-tiles for a single tile $F$. 
However, it turns out that additional structure on a $A \subset \Z^d$ can be concluded from the fact that it is a joint co-tile for $k$ tiles that are ``sufficiently different'', in a natural sense that in particular excludes the situations where these $k$ tiles cannot be obtained from a single tile by translations and dilations. 

As with ordinary systems of linear equations, it makes sense to introduce a notion of independence in this setup. For $F\Subset\Z^d$ we denote  
\[
F^* := F\setminus\{0\}.
\]

We say that $(F_1,\ldots,F_k)$ is an \emph{independent tuple} of tiles (or \emph{$k$ independent tiles}) if each $F_j$ is a finite subset of $\Z^d$, with $0 \in F_j$, and for every choice of $v_1 \in F_1^*,\ldots, v_k \in F_k^*$, the $k$-tuple $(v_1,\ldots,v_k)$ is independent (equivalently here, linearly independent vectors over $\mathbb{Q}$, or similarly over $\mathbb{R}$ or $\mathbb{C}$), see Figure \ref{pic:independent}. Notice that if  $(F_1,\ldots,F_k)$ is an independent tuple of tiles then $k \le d$. 
\begin{figure}[H]
\centering
\includegraphics[scale=0.5]{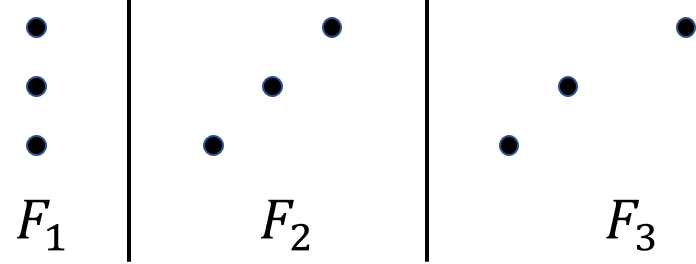}
\caption{Here $F_1, F_2, F_3\subset\Z^2$ and the lower point in all three of them is $(0,0)$. Each one of the pairs $(F_1,F_2)$ and $(F_1,F_3)$ is independent, but the pair $(F_2,F_3)$ is not. Also note that if the points in each of the sets lie on the lines $\{y=0\}$, $\{y=1\}$ and $\{y=2\}$, then $A=\Z\times 3\Z$ is a joint co-tile for $F_1, F_2,F_3$.}
\end{figure}\label{pic:independent}

Building on methods developed in \cite{BPeriodicity2020}, \cite{Greenfeld_Tao1_2021} and earlier work, we prove the following:
\begin{thm}\label{thm:periodic_decomposition}
For every $1\le k\le d$ the indicator function of any joint co-tile for $k$ independent tiles in $\Z^d$ is equal to a sum of $[0,1]$-valued $k$-periodic functions plus an integer constant.
\end{thm}

The case $k=1$ of \Cref{thm:periodic_decomposition} was proven in \cite{Greenfeld_Tao1_2021}. As a direct consequence of \Cref{thm:periodic_decomposition}, we obtain the following multidimensional generalization of Newman's result:
 
\begin{thm}\label{thm:d_independent_tiles}
For any independent tuple of tiles in $\Z^d$, the set of joint co-tiles is finite. In particular, 
any joint co-tile for $d$ independent tiles in $\Z^d$ is $d$-periodic. Furthermore, if $(F_1,\ldots,F_d)$ is an independent tuple of tiles in $\Z^d$ and $f:\Z^d\to \Z$ is a bounded function that satisfies $\ind_{F_i} * f = 1$ for every $1\le i \le d$, then $f$ is $d$-periodic.  
\end{thm}

We discuss further generalizations of Newman's theorem in \Cref{sec:countable_abelian} and particularly to groups of the form $\Z\times (\Z/p\Z)$ in \Cref{prop:newman_Z_mod_pZ}.

We say that a set $A \subseteq \Z^d$ is \emph{piecewise $k$-periodic} if  there exist $A_1,\ldots,A_r \subset \Z^d$ such that
$A= \biguplus_{j=1}^r A_j$ and each $A_j$ is $k$-periodic. 
Note that \cite{BPeriodicity2020}, \cite{Greenfeld_Tao1_2021} and \cite{Khetan_2021} used \emph{weakly periodic} for piecewise 1-periodic. 
In \cite{Greenfeld_Tao1_2021} it was shown that any $A \subseteq \Z^2$ satisfying $F \oplus A =\Z^2$ is piecewise $1$-periodic, whereas in \cite{BPeriodicity2020} it was shown that almost every solution to $F \oplus A = \Z^2$ is piecewise $1$-periodic, with respect to any invariant measure on the space of solutions. The apriori weaker ``almost everywhere'' result sufficed to prove the $\Z^2$ periodic tiling conjecture. The following result shows that the existence of piecewise $(d-1)$-periodic joint co-tiles implies the existence of $d$-periodic joint co-tiles. For $k=1$ and $d =2$ it coincides with the 
results in \cite{BPeriodicity2020}, \cite{Greenfeld_Tao1_2021}, deducing $2$-periodicity from piecewise $1$-periodicity.

\begin{thm}\label{thm:from_weakly_periodic_to_periodic}
Let $k$ and $d$ be positive integers and let $F_1,\ldots, F_k\Subset \Z^d$.
 If $F_1,\ldots, F_k$ admit a piecewise $(d-1)$-periodic joint co-tile,
    then they admit a $d$-periodic joint co-tile. 
\end{thm}

 
We now define an additional condition on a tuple of tiles, which is needed for the formulation of a certain generalization of Bhattacharya's and Greenfeld-Tao's  theorems to $d >2$:

\begin{definition}\label{def:property_star}
    Let $(F_1,\ldots,F_{d-1})$ be a tuple of tiles in $\Z^d$, $d \ge 2$. We say that $(F_1,\ldots,F_{d-1})$ has \emph{property $(\star)$} if it is an independent tuple and for every  $(v_1,\ldots,v_{d-1}),(w_1,\ldots,w_{d-1}) \in F_1^*\times \ldots \times F_{d-1}^*$ such that 
\[\spn(v_1,\ldots,v_{d-1})= \spn(w_1,\ldots,w_{d-1}),\]
we have $v_i=w_i$ for all $1 \le i \le d-2$.
\end{definition}

\begin{example} Consider the tiles
    \[F_1=\{(0,0,0),(1,0,0),(0,1,0),(1,1,0)\} \mbox{ and }
    F_2=\{(0,0,0),(1,0,1),(0,1,1),(1,1,1)\}.\]
    Then $(F_1,F_2)$ has property $(\star)$. Indeed it is easy to see that $(F_1,F_2)$ is an independent tuple of tiles in $\Z^3$, and for any $v_1 \in F_1^*$ and $v_2 \in F_2^*$ the intersection of $\spn(v_1,v_2)$ with the plane $V = \{ (a,b,0) ~:~ a,b \in \mathbb{R}\}$ is equal to $\spn(v_1)$.
    Also, for $A= 2\Z \times 2\Z \times \Z$ we have $F_1 \oplus A = F_2 \oplus A = \Z^3$, so $A$ is a (periodic) joint co-tile.
\end{example}



\begin{thm}\label{thm:main1}
Let $(F_1,\ldots,F_{d-1})$ be a tuple of tiles in $\Z^d$ that has property $(\star)$.
Then any joint co-tile for $F_1,\ldots,F_{d-1}$ is piecewise $(d-1)$-periodic.
\end{thm} The next statement follows immediately  from \Cref{thm:main1} together with \Cref{thm:from_weakly_periodic_to_periodic}.

\begin{cor}\label{cor:main_B}
Let $(F_1,\ldots,F_{d-1})$ be a tuple of tiles in $\Z^d$ that has property $(\star)$. If $(F_1,\ldots,F_{d-1})$ admits a joint co-tile then it admits a $d$-periodic joint co-tile.
\end{cor}

Note that for $d=2$, property $(\star)$ is vacuous, hence \Cref{thm:main1} reduces to the statement that any co-tile for a finite subset of $\Z^2$ is piecewise $1$-periodic (Greenfeld-Tao's theorem) and \Cref{cor:main_B} reduces to the statement that any finite subset of $\Z^2$ that admits a co-tile also admits a periodic co-tile (Bhattacharya's theorem). Hence for $d \geq 3$, it is natural to ask whether property $(\star)$ is a necessary condition for the existence of a periodic joint co-tile of $(d-1)$ tiles of $\Z^d$.

 We note a particular application of our methods, although not directly related to our main results:
\begin{thm}\label{thm:nice_corollary}
Suppose that $\Z^d$ decomposes into $(d-1)$-periodic subsets $A_1,\ldots,A_r \subset\Z^d$, where at least one of them is not $d$-periodic.
Then there exists $\Gamma\le\Z^d$ of rank $d-1$ so that $A_j+\Gamma=A_j$ for all $1\le j\le r$. 
\end{thm}

On the other hand, we obtain the following converse results for \Cref{thm:d_independent_tiles} and \Cref{cor:main_B}.  
\begin{thm}\label{thm:brother_tiles}
Suppose that $\{0\} \subsetneqq F \Subset \Z^d$ admits a periodic tiling $A \subseteq \Z^d$, then there exist $F_1,\ldots,F_{d-1} \Subset \Z^d$ with $0 \in F_j$ and $F_j \oplus A = \Z^d$ for all $ 1\le j \le d$, such that
\begin{enumerate} [label=(\alph*)]
    \item \label{thm_sec_a:brothers}
    $(F_1,\ldots,F_{d-1},F)$ is a $d$-tuple of independent tiles. 
    \item \label{thm_sec_b:brothers}
    $(F_1,\ldots,F_{d-2},F)$ has property $(\star)$.
\end{enumerate}
\end{thm}

Combining \Cref{cor:main_B} and \Cref{thm:brother_tiles} \ref{thm_sec_b:brothers} we obtain the following:

\begin{cor}\label{cor:equiv_condition_for_tile_periodically}
A finite set $\{0\} \subsetneqq F \Subset \Z^d$ tiles $\Z^d$ periodically if and only if there exists $F_1,\ldots,F_{d-2} \Subset \Z^d$ and $A \subset \Z^d$ such that $(F_1,\ldots,F_{d-2},F)$ has property $(\star)$, $F\oplus A = \Z^d$  and $F_j \oplus A = \Z^d$ for all $ 1\le j \le d-2$.
In particular, $F \Subset \Z^3$ tiles $\Z^3$ periodically if and only if there is another co-tile $F'$ for $A$ such that $(F',F)$ has property $(\star)$.
\end{cor}


The structure of the paper is as follows. \Cref{sec:preliminaries} contains basic background and definitions. 
In \Cref{sec:structure_theorem} we prove \Cref{thm:structure_theorem}, a periodic decomposition theorem for joint co-tiles, which is a refinement of \Cref{thm:periodic_decomposition}. From \Cref{thm:structure_theorem}, we directly deduce \Cref{thm:periodic_decomposition} and \Cref{thm:d_independent_tiles}. In \Cref{sec:countable_abelian}, we discuss generalizations of \Cref{thm:structure_theorem}, \Cref{thm:periodic_decomposition} and \Cref{thm:d_independent_tiles} to countable abelian groups. This allows us to extend Newman's Theorem to tilings of the group $\Z\times (\Z/p\Z)$. In \Cref{sec:strong_independency} we prove \Cref{thm:main1}, which asserts that property $(\star)$ implies piecewise $(d-1)$-periodicity of joint co-tiles. Then in \Cref{sec:weakly_piecewise_to_periodic} we prove \Cref{thm:nice_corollary} and deduce \Cref{thm:from_weakly_periodic_to_periodic}. \Cref{sec:constructing_brother_tiles} is dedicated to the proof of \Cref{thm:brother_tiles}. 
Finally, \Cref{sec:further_comments} contains concluding remarks and related questions. 

{\bf Acknowledgement.} We thank Itay Londner for discussions about tilings in cyclic groups and the Coven-Meyerowitz conditions. We thank Ilya Tyomkin for telling us about the relation between the dimension of the common complex zeros for a system of multivariate polynomials with integer coefficients, the tropical variety, and the associated Bieri-Groves set. We also thank Rachel Greenfeld and Terence Tao for their helpful communications and the anonymous referees for their remarks and suggestions. 

\section{Preliminaries}\label{sec:preliminaries} A function $f:\Z^d \to\R$ is called $L$-periodic, where $L\le\Z^d$, if for every $x\in\Z^d$ and $v\in L$ we have $f(x+v)=f(x)$. Recall that a set $A \subseteq \Z^d$ is \emph{piecewise $k$-periodic} if  $A$ is the disjoint union of $k$-periodic sets. 

\begin{definition}\label{def:poly}
Let $\Gamma_1,\Gamma_2$ be abelian groups.
For $f:\Gamma_1 \to \Gamma_2$ and $v \in \Gamma_1$, we define the \textit{discrete derivative of $f$ in direction $v$}, $D_v f:\Gamma_1 \to \Gamma_2$, by
\[ D_v f(w):= f(w) - f(w-v).\]
A function $P:\Gamma_1 \to \Gamma_2$ is called \emph{a polynomial map of degree at most $r$} if 
\[ ~\forall\, v_1,\ldots,v_{r+1} \in \Gamma_1:\quad D_{v_1}\ldots D_{v_{r+1}} P = 0\]
(where for consistency $P = 0$, is a polynomial of degree $-1$).
Given a subgroup $\Gamma_3 < \Gamma_1$, one says that  $P:\Gamma_1 \to \Gamma_2$ is a \emph{polynomial map of degree at most $r$ with respect to $\Gamma_3$} if 
\[ ~\forall\, v_1,\ldots,v_{r+1} \in \Gamma_3:\quad D_{v_1}\ldots D_{v_{r+1}} P = 0.\]
\end{definition}

The following basic facts about polynomials will be useful for us. \Cref{lem:bounded_poly=constant} below is due to Leibman \cite[Prop. 1.21]{Leibman2002}. We include a short proof for the reader's convenience. 
\begin{lemma}\label{lem:bounded_poly=constant}
Let $P:\Z^d \to \R$ be a polynomial map with respect to a finite index subgroup $L\le\Z^d$, which is bounded, then $P$ is constant on cosets of $L$.
\end{lemma}
\begin{proof}
Let $r\in\N$ denote the degree of $P$, as a polynomial with respect to $L$. It is clear from  \Cref{def:poly} that if $r$ is equal to $0$, then the restriction of $P$ to each coset of $L$ is a constant. Similarly, if the degree of $P$ is equal to $1$, then the restriction of $P$ to each coset of $L$ is a constant plus a non-trivial homomorphism (see e.g. \cite{Leibman2002}). For contradiction, we may assume that $r \ge 1$. Observe that since $P$ is bounded and for every $v\in L$ we have $D_vP\subseteq P(\Z^d)-P(\Z^d)$, thus $D_vP$ is bounded. Therefore, for every $v_1,\ldots,v_{r-1}\in L$ the function $D_{v_1}\ldots D_{v_{r-1}}P$ is a bounded polynomial map of degree exactly one, with respect to $L$. But non-trivial homomorphisms into $\R$ are unbounded, a contradiction.    
\end{proof}
\begin{definition}\label{def:mean}
We say that a bounded function $f:\Z^d \to \R$ has mean $m$ if  
\begin{equation}\label{eq:mean}
\lim_{n\to\infty}\frac{1}{|B_n|}\sum_{v\in B_n}f(v)=m,
\end{equation}
where $B_n = \{-n,\ldots,n\}^d$.

One says that $f:\Z^d \to \R/\Z$ is equidistributed in $\R/\Z$ if 
\begin{equation}\label{eq:equidstirbuted}
\lim_{n\to\infty}\frac{1}{|B_n|}\sum_{v\in B_n}g(f(v))=\int_0^1 g(x) dx
\end{equation}
holds for every continuous function $g:\R/\Z \to \R$, where we identify $g:\R/\Z \to \R$ with $g:\R \to \R$ such that $g(x+1)=x$ for all $x \in \R$.
\end{definition}
We will use the following version of Weyl's equidistribution theorem for polynomials in several variables \cite[Theorem 20]{MR1511862}.  See also  \cite[Chapter 1.1]{MR2931680} and \cite[Theorem 1.2]{Yiftah_2022}. 

\begin{thm}[Weyl's equidistribution theorem for polynomials in several variables]\label{thm:weyl_equidist_zd}
Let $P:\Z^d \to \R/\Z$ be a polynomial map with respect to a finite index subgroup $\Gamma$ of $\Z^d$. Then on every coset $v+ \Gamma$ of $\Gamma$, the restriction of $P$ to $v+ \Gamma$ is either equidistributed in $\R/\Z$ or periodic.
\end{thm}

\ignore{
\begin{definition}
If $F_1,\ldots,F_k \Subset \Z^d$ each contain $0$, we say that  $F_1,\ldots,F_k \Subset \Z^d$ are \emph{independent tiles} if for every choice of $v_1 \in F_1^*,\ldots, v_k \in F_k^*$, the set $\{v_1,\ldots,v_k\}$ is linearly independent (as vectors over $\mathbb{Q}$, or equivalently over $\mathbb{R}$ or $\mathbb{C}$). 
Obviously, if  $F_1,\ldots,F_k \Subset \Z^d$  are independent then $k \le d$.
\end{definition}

\begin{definition}\label{def:problematic}
Let $F_1,\ldots,F_{d-1} \Subset \Z^d$ be independent tiles such that
\[ |F_1|= \ldots =|F_{d-1}| = S.\]
    We say that the collection $\{F_1,\ldots,F_{d-1}\}$ is \emph{problematic} if there exists a $(d-1)$-dimensional subspace $V \subseteq \R^d$ such that
\begin{equation}
    \left|
\left\{
(v_1,\dots,v_{d-1}): \forall 1\le i \le d-1, v_i \in F_i^* \cap V
\right\}
\right| \ge \frac{S}{2}
\end{equation}
\end{definition}
}


We implicitly rely on the following basic observation (compare \cite[Lemma 1.1]{MR2087242}).
\begin{prop}\label{prop:mean}
Let $F \Subset \Z^d$. Suppose that $F\subset B_{n_0}$ for some $n_0\in\N$ and that $f:\Z^d\to\R$ is a bounded function satisfying
$\ind_F * f = 1$. 
Denote by $C = |F|(\max f - \min f)$. 
Then for every $n>n_0$ one has 
\begin{equation}\label{eq:co_tile_density_estimates}
 |B_{n-n_0}| - C\absolute{B_{n+n_0}\setminus B_{n-n_0}} \le |F|\sum_{w \in B_n}f(w) \le |B_{n-n_0}|+ C\absolute{B_{n+n_0}\setminus B_{n-n_0}},
\end{equation}
and thus the function $f$ has mean $\frac{1}{|F|}$.
In particular, if $F_1, F_2 \Subset \Z^d$ satisfy 
$\ind_{F_1} * f = \ind_{F_2} * f = 1$,
then $|F_1|=|F_2|$. 
\end{prop}

\begin{proof}
Pick $n_0\in\N$ such that $F \subset B_{n_0}$. Observe that $\ind_F * f = 1$ implies that for every $n> n_0$ we have 
\[
\ind_{B_{n-n_0}} - C\cdot \ind_{B_{n + n_0}\setminus B_{n-n_0}}\le \ind_F * f|_{B_n} \le \ind_{B_{n-n_0}} + C\cdot \ind_{B_{n + n_0}\setminus B_{n-n_0}},
\]
where $f|_{B_n}$ denotes the restriction of $f$ to $B_n$. Taking the sum of the values of these functions over all $z\in\Z^d$ implies that  \eqref{eq:co_tile_density_estimates} holds for every $n>n_0$. Since $\lim_{n \to \infty} \frac{|B_{n - n_0}|}{|B_n|}=1$ and $\lim_{n \to \infty} \frac{|B_{n + n_0}\setminus B_{n-n_0}|}{|B_n|} = 0$, dividing \eqref{eq:co_tile_density_estimates} by $\absolute{F}\cdot\absolute{B_n}$ and letting $n \to \infty$ yields the assertion. 
\end{proof}

\begin{remark}
The mean of a function $f:\Gamma\to\R$ is defined similarly, using \eqref{eq:mean}, for any countable amenable group $\Gamma$, in which case $B_n$ is replaced by a  F\o lner sequence in $\Gamma$, and an analogue of \Cref{prop:mean} holds in this more general context as well. In \Cref{sec:further_comments}, we implicitly apply \Cref{prop:mean} for countable abelian groups $\Gamma$, which are in particular amenable.
\end{remark}

\subsection{Shifts of finite type} The space of co-tiles for a given finite set $F \subset \Z^d$, or more generally, the space of joint co-tiles for a given collection of sets, can naturally be endued with the structure of a compact topological space on which $\Z^d$ acts by homeomorphisms. Topological dynamical systems of this kind are called \emph{$\Z^d$-subshifts}, more specifically \emph{subshifts of finite type}. 
We include here relevant terminology and basic facts from the field of symbolic dynamics, particularly regarding shifts of finite type. We refer to \cite{LindMarcus} for a comprehensive introduction to symbolic dynamics.   

Let $\Sigma$ be a finite set (alphabet) and $\Gamma$ a finitely generated abelian group. The set of functions from $\Gamma$ to $\Sigma$, denoted $\Sigma^\Gamma$, is called \emph{the full $\Gamma$-shift}.
For $x \in \Sigma^\Gamma$ and $v \in \Gamma$, we use $x_v$ to denote the value of $x$ at $v$ (this is an element of $\Sigma$). Also for $x \in \Sigma^\Gamma$ and $v \in \Gamma$ we denote by $\sigma_v(x) \in \Sigma^\Gamma$ the \emph{shift of $x$ by $v$}, which is given by 
\[
\sigma_v(x)_w = x_{v+w}.
\]
Endowing $\Sigma^\Gamma$ with the product topology, where the topology on $\Sigma$ is the discrete topology, makes $\Sigma^\Gamma$ a compact $\Gamma$-space. 
A closed, non-empty and $\Gamma$-invariant subset $X\subseteq \Sigma^\Gamma$ is called a \emph{$\Gamma$-subshift}. 
For $x \in \Sigma^\Gamma$, the \emph{stabilizer of $x$} is defined to be
\[
\stab(x) = \{ v \in \Gamma:~ \sigma_v(x)=x\},
\]
which is a (possibly trivial) subgroup of $\Gamma$. 
A point $x\in\Sigma^\Gamma$ is called \emph{$k$-periodic} if $\stab(x)$ is a subgroup of rank $k$. When $\Gamma=\Z$, we say that  $x \in \Sigma^\Z$ is \emph{periodic} if it has a non-trivial stabilizer.

\begin{definition}\label{def:SFT}
A $\Gamma$-subshift $X \subseteq \Sigma^\Gamma$ is called \emph{a subshift of finite type (SFT)} if there exists a finite set $W \subset \Gamma$ and a set $\mathcal{F} \subseteq \Sigma^W$ such that
\[ X = \left\{ x \in \Sigma^\Gamma:~ \forall v \in \Gamma, ~ \sigma_v(x)|_W \not \in \mathcal{F}  \right\}.\]
\end{definition}

For every $F \Subset \Z^d$ the space of co-tiles for $F$ is a subshift of finite type, under the natural identification of the space of co-tiles for $F$ with
\[X_F := \left\{ x \in \{0,1\}^{\Z^d}:~ \ind_F * x = 1\right\}.\]

To see that $X_F$ is indeed an SFT, take $W= -F$ and 
\[ \mathcal{F} = \left\{ p \in \{0,1\}^W ~:~ \sum_{w \in W} p(w) \ne 1 \right\},\]
and then 
\[
X_F = \left\{ x \in \{0,1\}^{\Z^d} ~:~ \forall v\in\Z^d, ~ \sigma_v(x)|_W \not \in \mathcal{F} \right\}.
\]
Since a non-empty intersection of SFTs is also an SFT, it follows that the space of joint co-tiles for a collection of tiles is an SFT (unless it is empty).

The following simple result is based on a pigeonhole argument. The proof is well-known and standard, see for instance \cite[Exercise 2.2.10]{LindMarcus}. We include a proof for completeness.

\begin{lemma}\label{lem:Z_shift_periodic}
Every $\Z$-subshift of finite type admits a periodic point.
\end{lemma}
\begin{proof}
Let $X \subseteq \Sigma^\Z$ be a $\Z$-subshift of finite type, where 
$\Sigma$ is  a finite set.
Then by definition, there exists a finite set $W \Subset \Z$ and $\mathcal{F} \subseteq \Sigma^W$ such that 
\[ X = \left\{ x \in \Sigma^\Z:~ \forall v \in \Z, ~ \sigma_v(x)|_W \not \in \mathcal{F}  \right\},\]
and $X\neq \emptyset$. Fix $x \in X$, and let $N \in \mathbb{N}$ be an integer bigger than $\max(W) - \min(W)$.
Since the set $\Sigma^{\{1,\ldots,N\}}$ is finite, by the pigeonhole principle there exist integers $0 \le i < j \le |\Sigma|^N$ such that
\[x|_{\{i,\ldots,i+N-1\}} = x|_{\{j,\ldots,j+N-1\}}.\]
Let $p=j-i$ and define $\hat x \in \Sigma^\Z$ by
\[
\hat x_n = x_{i+ (n\mod p)}.
\]
Then $\hat x$ is a periodic point, and for every $n \in \Z$ there exists $t \in \{i,\ldots,j-1\}$ such that $\hat x|_{W+n} = x|_{W+t}$. Hence, $\hat x \in X$, which proves that $X$ admits a periodic point.
\end{proof}

We recall the following 
result in multidimensional symbolic dynamics.
\begin{lemma}\label{lem:stab_SFT}
Let $\Gamma$ be a finitely generated abelian group, $\Gamma_0 \le \Gamma$ a subgroup,  and $X \subseteq \Sigma^\Gamma$ a $\Gamma$-subshift. 
Let 
\begin{equation}\label{eq:X_Gamma_0}
    X_{\Gamma_0} := \left\{ x \in X~: \Gamma_0 \le \stab(x) \right\}.
\end{equation}
If $X_{\Gamma_0} \neq \emptyset$ then it is a $\Gamma$-subshift. Furthermore, if $X$ is a subshift of finite type then $X_{\Gamma_0}$ is also a subshift of finite type.
\end{lemma}

\begin{proof} 
First, we show that $X_{\Gamma_0}$ is a subshift. Since $\Gamma$ is abelian, for every $v \in \Gamma$, $v_0 \in \Gamma_0$ and $y \in X_{\Gamma_0}$ we have
\[
\sigma_{v_0} (\sigma_v (y)) = \sigma_{v} (\sigma_{v_0} (y)) = \sigma_{v} (y).
\]
This shows $\sigma_v (y) \in X_{\Gamma_0}$ for all $v \in \Gamma$ hence $X_{\Gamma_0}$ is $\Gamma$-invariant. To see that $X_{\Gamma_0}$ is a closed subset of $\Sigma^{\Gamma}$, consider a sequence $(y_n)_{n \in \N} \in X_{\Gamma_0}$ such that
\[
\underset{n \rightarrow \infty}{\mathrm{lim}}\,\, y_n = y \in \Sigma^{\Gamma}
\]in the product topology. Since each $y_n \in X_{\Gamma_0} \subseteq X$ and $X$ is a closed subset of $\Sigma^{\Gamma}$, we get $y \in X$. Note that for any $v_0 \in \Gamma_0$,
$$
\sigma_{v_0} (y) = \sigma_{v_0} \,\,\big(\underset{n \rightarrow \infty}{\mathrm{lim}} y_n\big) = \underset{n \rightarrow \infty}{\mathrm{lim}} (\sigma_{v_0} (y_n)) = \underset{n \rightarrow \infty}{\mathrm{lim}} (y_n) = y,
$$ which shows $y \in X_{\Gamma_0}$ and hence $X_{\Gamma_0}$ is a subshift. 
Now assuming that $X$ is an SFT we show that $X_{\Gamma_0}$ is also an SFT. Observe that $X_{\Gamma_0} = X \cap Y$ where 
\[
Y = \{x \in \Sigma^{\Gamma} ~:~ \Gamma_0 \le \stab(x)  \}.
\]
Since $\Gamma_0$ is a subgroup of a finitely generated abelian group it is also finitely generated.
Let $\{\gamma_1,\ldots,\gamma_r\}$ be a finite generating set for $\Gamma_0$.
Then
\[
Y = \bigcap_{i=1}^r \{x \in \Sigma^{\Gamma} ~:~  \forall v \in \Gamma, ~  x_{v + \gamma_i} = x_v    \}.
\]
To see that $Y$ is an SFT, let $W= \{0,\gamma_1,\ldots,\gamma_r\}$ and 
\[
\mathcal{F} = \left\{ w \in \Sigma^W:~ \exists 1 \le i \le r \mbox{ s.t. } w_0 \ne w_{\gamma_i}\right\}.
\]
Then
\[ Y = \left\{ x \in \Sigma^\Gamma ~:~ \forall v \in \Gamma, ~ \sigma_v(x)|_{W} \notin \mathcal{F}  \right\}. \] 
Hence $Y$ is an SFT, which completes the argument. 

\end{proof}

From \Cref{lem:stab_SFT} we deduce the following conclusion:
\begin{lemma}\label{lem:Zd_SFT_periodic_points_from_d_minus_1}
Let $\Gamma$ be a finitely generated abelian group of rank $d$.
If $X \subseteq \Sigma^{\Gamma}$ is a $\Gamma$-subshift of finite type that admits a $(d-1)$-periodic point then it admits a $d$-periodic point.
\end{lemma}
\begin{proof}
Suppose $X \subseteq \Sigma^{\Gamma}$ is a $\Gamma$-subshift of finite type that admits a $(d-1)$-periodic point, namely a point $z \in X$ and a subgroup $\Gamma_0 \le \Gamma$ of rank $d-1$ such that 
$\stab(z) =\Gamma_0$.
Let $X_{\Gamma_0}$ be given by \eqref{eq:X_Gamma_0}.
Then  $X_{\Gamma_0}$ is non-empty, and by \Cref{lem:stab_SFT} it is a subshift of finite type.
Because $\rank(\Gamma_0) = d-1$, it follows that $\rank(\Gamma / \Gamma_0) =1$.
Let $v \in \Gamma$ be a vector such that $k\cdot v \not \in \Gamma_0$ for all $k \in \N$.
Then $\Gamma_0 \oplus \Z v$ is a finite index subgroup of $\Gamma$. Let $D\subseteq \Gamma$ be a fundamental domain for $\Gamma_0 \oplus \Z v$, namely a  finite set such that $\Gamma_0 \oplus \Z v  \oplus D = \Gamma$. 
Because $D \oplus \Z v$ is a fundamental domain for $\Gamma_0$ in $\Gamma$, it follows that the restriction map $\rho: X_{\Gamma_{0}} \to \Sigma^{D \oplus \Z v}$ is injective, where $\rho$ is given by
 $\rho(x) = x\mid_{D \oplus \Z v}$.

Indeed, the inverse $\rho^{-1}:\rho(X_{\Gamma_{0}}) \to X_{\Gamma_0}$ is given by
$\rho^{-1}(\tilde x)_{u} = (\tilde{x})_{u'}$  for $u \in \Gamma$,
where $u'$ is the unique element in  $(D \oplus \Z v)$ that satisfies  $u-u' \in \Gamma_0$.
Using the natural identification  $\Sigma^{D \oplus \Z v} \cong (\Sigma^D)^{\Z}$,
we can view $\rho(X_{\Gamma_0})$ as a subset of $ (\Sigma^D)^{\Z}$, which we denote by $\tilde X$. 

Let us show that $\tilde X$ is a $\Z$-subshift of finite type.
Because $X$ is a $\Gamma$-subshift of finite type, there exists a finite set $W \subset \Gamma$ and $\mathcal{F} \subset \Sigma^W$ such that 
$X_{\Gamma_0}$
is equal to the set of $x \in \Sigma^\Gamma$ satisfying $\sigma_v(x)=x$ and  
$\sigma_v(x)\mid_{W} \not \in \mathcal{F}$ for all $v \in \Gamma_0$. 
We can assume without loss of generality that $W$ is a subset of $\mathbb{Z} v \oplus D$,
because $\Z v \oplus D$ is a fundamental domain for $\Gamma_0$. 
Let $\tilde W = \{ n \in \Z:~ (n v + D) \cap W \ne \emptyset\}$. 
Then $W = \biguplus_{ n \tilde W} (W \cap (nv +D))$. 
Thus, there is a natural bijection between $\Sigma^W$ and $(\Sigma^D)^{\tilde W}$. 
Let $\tilde{\mathcal{F}}$ 
denote the image of $\mathcal{F}$ under this bijection.
Then it follows directly that
\[\tilde X = \left\{ x \in (\Sigma^D)^\Z:~ \forall v \in \Z:~ \sigma_v(x)\mid_{\tilde W} \not \in \tilde{\mathcal{F}} \right\}.\]
This proves that $\tilde X$ is indeed a $\Z$-subshift of finite type.

Since  $\tilde X$ is a $\Z$-subshift of finite type, by \Cref{lem:Z_shift_periodic} there exists a periodic point $\tilde z$ in $\tilde X$. Let $x = \rho^{-1}(\tilde z)$, then $x \in X$ is a $d$-periodic point.
\end{proof}

 \begin{remark}
    The simple argument behind \Cref{lem:Zd_SFT_periodic_points_from_d_minus_1} has been applied in \cite{BPeriodicity2020} for the case $\Gamma = \Z^2$, and also in a number of older references. See for instance  \cite{MR1307966} and \cite{MR1428636}. We had difficulties finding a precise reference to the statement of \Cref{lem:Zd_SFT_periodic_points_from_d_minus_1} in  full generality.
 \end{remark}
\section{The periodic decomposition theorem}\label{sec:structure_theorem}

The following theorem asserts a certain decomposition for a joint co-tile of $k$-tuple of tiles in $\Z^d$. The case where $k=1$ and $f$ is $\{0,1\}$-valued essentially coincides with \cite[Theorem 1.7]{Greenfeld_Tao1_2021}, which is closely related to  \cite[Theorem 3.3]{BPeriodicity2020}. 
In the particular case that the tuple of tiles is independent, \Cref{thm:periodic_decomposition} is a direct consequence. Namely, the indicator function of any joint co-tile of $k$ independent
tiles is a sum of $k$-periodic functions, each taking values in $[0,1]$. The goal of this section is to prove the periodic decomposition theorem for joint co-tiles and to deduce \Cref{thm:periodic_decomposition} and \Cref{thm:d_independent_tiles}.

\begin{thm}[Periodic decomposition theorem]
\label{thm:structure_theorem} Let $F_1,\ldots,F_k \Subset \Z^{d}$, with $0 \in F_i$ for all $1 \le i \le k$, and let $f:\Z^d\to\Z$ be a bounded function that satisfies $\ind_{F_i} * f = 1$ for all $1 \le i \le k$. We denote by $S : =|F_1|= \ldots =|F_k|$ (see \Cref{prop:mean}). Then for every $1 \le i \le k$ and every $(v_1,\ldots,v_i) \in F_1^*\times \ldots \times F_i^*$ there exists a function $\phi_{v_1,\ldots,v_i}:\Z^d \to [\min f,\max f]$ with the following properties:
\begin{enumerate}[label=(\alph*)]
\item\label{thm_sec_a:structure_theorem}  
For $i < k$ we have
\begin{equation*}
        \phi_{v_1,\ldots,v_i} = 1 - \sum_{v_{i+1} \in F_{i+1}^*} \phi_{v_1,\ldots,v_i,v_{i+1}}.
\end{equation*}
\item\label{thm_sec_b:structure_theorem}  
\begin{equation*}
        f = (-1)^i\sum_{(v_1,\ldots,v_i) \in F_1^*\times \ldots \times F_i^*} \phi_{v_1,\ldots,v_i} +\sum_{j=1}^{i}(-(S-1))^{j-1}.
\end{equation*}
\item\label{thm_sec_c:structure_theorem} 
Let $q$ denote the product of all primes less than or equal to $(\max f - \min f)S$, then
\[
\left( \Z qv_1 + \ldots+ \Z qv_i \right)  \le  \stab(\phi_{v_1,\ldots,v_i}),
\]
\item\label{thm_sec_d:structure_theorem} 
$1_{F_j} *\phi_{v_1,\ldots,v_i}=1$ for all $1 \le j \le k$. In particular, $\phi_{v_1,\ldots,v_i}$ has mean $1/S$.
\end{enumerate}
\end{thm}

There are various extensions of \Cref{thm:structure_theorem}.  Some of these generalizations have further applications. For the sake of readability, we do not state the most general form and instead indicate certain generalizations in the following sections, at the expense of some repetition.

The proof of \Cref{thm:structure_theorem} relies on \Cref{lem:dilation} below. Various versions of this lemma, which is referred to as the dilation lemma, have been proved e.g. in \cite[Lemma 3.1]{Greenfeld_Tao1_2021}, \cite[Proposition 3.1]{BPeriodicity2020} for $\Gamma = \Z^d$, $d \geq 1$. We also refer our readers to \cite[Theorem 1]{Tijdeman_1995} where this lemma is proved for integers. The proof is based on some elementary commutative algebra and it easily extends to countable abelian groups. For the sake of self-containment, we include a sketch of the proof below. The proof below is nearly identical to \cite[Lemma 3.1]{Greenfeld_Tao1_2021}, except that we apply the assumption that $r$ is co-prime to the order of torsion elements directly before \cref{eq:Dilation_extends}.

\begin{lemma}[Dilation lemma]\label{lem:dilation} 
Let $\Gamma$ be a countable abelian group. Let $0\in F  \Subset \Gamma$, $\ell\in \N$ and $f:\Gamma \to \Z$ a bounded function satisfying 
\[
\ind_F * f = \ell.
\] 
Let $q_1$ be the product of all primes less than or equal to $(\max f - \min f)|F|$, let $q_2$ be the product of all the orders of the torsion elements in $(F-F)$, and set $q = q_1q_2$. Then 
\[
\ind_{rF} * f = \ell,
\] 
for all $r \in \N$ such that $r = 1 \mod q$.
\end{lemma}

\begin{proof}
    We use the notation $f^{*p} = \underbrace{f*\ldots*f}_{\times p}$.  
    For any prime $p$ we have 
    \[\ind_F^{*p} = \left( \sum_{v\in F} \delta_v \right)^{*p} = \sum_{v\in F} \delta_v^{*p} \mod p,\]
where the last equality holds by the Frobenius identity $(f+g)^{*p} = f^{*p} + g^{*p} \mod p$. For integers $p$ that are co-prime to $q_2$ we have that $p(v_1-v_2) \ne 0$ for any $v_1\ne v_2 \in F$, so the function $v \mapsto pv$ is injective on $F$. Thus:
\begin{equation}\label{eq:Dilation_extends}
    \sum_{v\in F} \delta_v^{*p} = \sum_{v\in F} \delta_{pv} = \ind_{pF}.
\end{equation}

Now convolving both sides of $\ind_F * f =\ell$ by $\ind_F^{*(p-1)}$ yields $\ind_F^{*p} * f = \ell |F|^{p-1}$. Combining the above, for primes $p$ that are co-prime to $q_2$ we obtain $\ind_{pF} * f = \ell |F|^{p-1} \mod p$. If additionally $p$ is co-prime to $|F|$ by Fermat little theorem $|F|^{p-1} = 1 \mod p$, thus 
\[ \ind_{pF} * f = \ell \mod p.\]
 Note that both $\ind_{F} * f$ and $ \ind_{pF} * f$ take values in $[|F|\min f, |F|\max f]$. Recall that $\ell = \ind_F * f$, so $\ell  \in [|F|\min f, |F|\max f]$. Thus,  for $p$ that is also greater than the size of that interval, the above equality holds without the $\mod p$, namely $\ind_{pF} * f = \ell$. Finally, for $r = 1 \mod q$, $r$ is a product of primes that satisfy the conditions above, and the result follows by iterating the equation $\ind_{pF} * f = \ell$. 
\end{proof}

\begin{proof}[Proof of \Cref{thm:structure_theorem}]  
For $1 \le i \le k$, $(v_1,\ldots,v_i)  \in F_1^* \times \ldots \times F_i^*$ and $N \in \N$ denote:
\begin{equation}\label{eq:phi_i^N}
\phi^{(N)}_{v_1,\ldots,v_i} := \dfrac{1}{N^i} \sum_{n_1,\ldots,n_i = 1}^{N} \delta_{(1+n_1 q)v_1+\ldots+(1+n_i q)v_i} * f.
\end{equation}
Let $q$ be the product of all primes less than or equal to $(\max f - \min f)S$. 
By applying \Cref{lem:dilation} for $F_j$ with $\Gamma = \Z^d$ and $\ell=1$ we get $\ind_{rF_j} * f = 1$ for every $r \in q\N+1$. Since $0 \in F_j$ we obtain 
\[
f = 1 - \sum_{v \in F^*_j} \delta_{rv}* f \mbox{ for every } 1 \le j \le k.
\] 
For every $N\in\N$, setting $r = 1+n q$ for $n \in \{1,\ldots,N\}$ and taking average we conclude that for every $1 \le j \le k$ we have 
\begin{equation}\label{eq:avg1}
    f = 1 -  \sum_{v \in F^*_j}\dfrac{1}{N} \sum_{n = 1}^{N} \delta_{(1+n q)v} * f.
\end{equation} 
Since $\phi_{v_1}^{(N)} = \frac{1}{N}\sum_{n=1}^N \delta_{(1+nq)v_1}* f$ this gives (with $j=1$):
\begin{equation}\label{eq:basis_of_induction}
    f = 1 - \sum_{v_1 \in F^*_1} \phi^{(N)}_{v_1}.
\end{equation}

For $1\le i<k$, choose any $(v_1,\ldots,v_i) \in F_1^*\times\ldots\times F_i^*$ and $1\le n_1,\ldots,n_i\le N$.
Setting $j = i+1$ in \eqref{eq:avg1} and convolving both sides of the  equation  by $\delta_{(1 + n_1 q) v_1 + \ldots + (1 + n_i q) v_i}$ we obtain
\begin{equation*}
    \delta_{(1 + n_1 q) v_1 + \ldots + (1 + n_i q) v_i} * f = 1 -  \sum_{v_{i+1} \in F^*_{i+1}} \dfrac{1}{N} \sum_{n_{i+1} = 1}^{N} \delta_{(1+n_1 q)v_1 + \ldots + (1 + n_i q) v_i + (1 + n_{i+1} q) v_{i+1}} * f.
\end{equation*} 
By averaging over $1 \le n_1,\ldots,n_i \le N$ and applying the definition in \eqref{eq:phi_i^N} we obtain that
\begin{equation}\label{eq:phi_rec}
    \phi^{(N)}_{v_1,\ldots,v_i} = 1 -  \sum_{v_{i+1} \in F^*_{i+1}} \dfrac{1}{N^{i+1}} \sum_{n_1,\ldots,n_{i+1} = 1}^{N} \delta_{(1+n_1 q)v_1 + \ldots + (1+n_{i+1} q)v_{i+1}} * f = 
    1 - \sum_{v_{i+1}\in F^*_{i+1}}\phi^{(N)}_{v_1,\ldots,v_{i+1}}.
\end{equation} 

Since $|F^*_i| = S-1$ for $ 1\le i \le k$, using \eqref{eq:basis_of_induction}, \eqref{eq:phi_rec} and an inductive argument we obtain that for every $N \in \N$ and $1\le i \le k$ we have
\begin{equation}\label{eq:property_b}
    f = \sum_{j=1}^i (-(S-1))^{j-1} + (-1)^i\sum_{(v_1,\ldots,v_i) \in F_1^*\times \ldots \times F_i^*} \phi^{(N)}_{v_1,\ldots,v_i}
\end{equation} 

Notice that the functions $\delta_{(1+n_1 q)v_1 +\ldots+ (1+n_iq)v_i} * f$ are bounded between $\min f$ and $\max f$, thus by \eqref{eq:phi_i^N}, the functions  $\phi^{(N)}_{v_1,\ldots,v_i}$ are bounded between $\min f$ and $\max f$ for every $(v_1,\ldots,v_i) \in F_1^* \times \ldots \times F_i^*$. In particular,  for every $ (v_1,\ldots,v_i) \in F^*_1\times\ldots\times F^*_i$ the sequence of functions $(\phi^{(N)}_{v_1,\ldots,v_i})_{N \in \N}$ is uniformly bounded, hence by Arzelà–Ascoli theorem (or by a Cantor diagonalization argument), it converges along a subsequence. We denote the limit by $\phi_{v_1,\ldots,v_i}$. Then for every $(v_1,\ldots,v_i) \in F^*_1\times\ldots\times F^*_i$ we have 
\[\min f \le \phi_{v_1,\ldots,v_i} \le \max f,\] 
and in view of \eqref{eq:phi_rec} and \eqref{eq:property_b} we have achieved \ref{thm_sec_a:structure_theorem} and \ref{thm_sec_b:structure_theorem}.

To see \ref{thm_sec_c:structure_theorem}, using \eqref{eq:phi_i^N}, a standard telescoping argument shows that for every $w \in \Z^d$, $v = (v_1,\ldots,v_i) \in F^*_1\times \ldots \times F^*_i$ and every $1\le j\le i$ 
\begin{equation*}\label{eq:telescope}
    \absolute{\phi^{(N)}_{v_1,\ldots,v_i}(w+ q v_j) - \phi^{(N)}_{v_1,\ldots,v_i}(w)} \le \frac{2N^{i-1}}{N^i} \left( \max(f)- \min(f) \right)= \frac2N \left( \max(f)- \min(f) \right) .
\end{equation*} 
Thus for every $(v_1,\ldots,v_i) \in F^*_1\times \ldots \times F^*_i$ the function $\phi_{v_1,\ldots,v_i}$ is $qv_j$-periodic for every $1\le j\le i$. It remains to show \ref{thm_sec_d:structure_theorem}. 
Clearly, since $\ind_{F_j} * f = 1$, for every $1\le i,j\le k$, $(v_1,\ldots,v_i) \in F_1^* \times \ldots \times F_i^*$ and $n_1,\ldots,n_i \in \N$ we have $\ind_{F_j}*(\delta_{(1+n_1q)v_1+\ldots(1+n_iq)v_i} * f) = 1$. Thus, by \eqref{eq:phi_i^N}, $\ind_{F_j} * \phi_{v_1,\ldots,v_i}^{(N)} = 1$ for every $N\in\N$ and therefore $\ind_{F_j} * \phi_{v_1,\ldots,v_i} = 1$ for every $1\le i, j\le k$. 
In particular, by  \Cref{prop:mean}, $\phi_{v_1,\ldots,v_i}$ has mean $1/S$. 
\end{proof}

\begin{remark}\label{remark:stracture_theorem_extra}
Under the assumption that $f$ is $\{0,1\}$-valued, it directly follows from \cref{thm:structure_theorem}, part \ref{thm_sec_a:structure_theorem}, that for every $1\le i <k$ and every $(v_1,\ldots,v_i)\in F_1^*\times\ldots\times F_i^*$, the sum $\sum_{v_{i+1} \in F^*_{i+1}} \phi_{v_1,\ldots,v_i,v_{i+1}}$ is a $[0,1]$-valued function. 
\Cref{thm:structure_theorem}, where $k=1$ and $f$ is $\{0,1\}$-valued, coincides with \cite[Theorem 1.7]{Greenfeld_Tao1_2021}.
We will not make use of the  property that $\ind_{F_j} * \phi_{v_1,\ldots,v_i}=1$ in this paper. We mention it only for completeness and possibly for future reference. The fact that the functions $\phi_{v_1}$ each have mean $1/S$ played an implicit role in \cite{BPeriodicity2020}.
\end{remark}

Using the assumption that the tuple of tiles is independent \Cref{thm:periodic_decomposition} is an immediate corollary of \Cref{thm:structure_theorem}, with $f$ being a $\{0,1\}$-valued function. The proof of  \Cref{thm:d_independent_tiles} is straightforward.
\begin{proof}[Proof of \Cref{thm:d_independent_tiles}]
Suppose that $(F_1,\ldots,F_d)$ is an independent tuple of tiles in $\Z^d$ and that $f:\Z^d\to\Z$ is a bounded function satisfying $\ind_{F_i} * f = 1$ for all $1\le i\le d$. 
By \Cref{prop:mean}, we have $\absolute{F_1}=\ldots=\absolute{F_d}:=S$. Let  $q$ be the product of all primes less than or equal to $(\max f - \min f)S$ and let
\[ 
L = \bigcap_{(v_1,\ldots,v_d)  \in F_1^* \times \ldots\times F_d^*} q \Z v_1 +  \ldots+  q \Z v_d.
\] 
Apply \Cref{thm:structure_theorem} with $k=d$. It follows that $f$ is a sum of functions whose stabilizers are rank $d$-subgroups,
more precisely, 
\[ f = (-1)^d \sum_{(v_1,\ldots,v_d) \in F_1^* \times \ldots \times F_d^*}\phi_{v_1,\ldots,v_d} + \sum_{j=1}^{d}(-(S-1))^{j-1},\]
and for each $(v_1,\dots,v_d) \in  F_1^* \times \ldots \times F_d^*$ we have that
$q\Z v_1+  \ldots + q \Z v_d \le \stab(\phi_{v_1,\ldots,v_d})$.
By the above, $\stab(f)$ contains the intersection of $\stab(\phi_{v_1,\ldots,v_d})$ over  $(v_1,\dots,v_d) \in  F_1^* \times \ldots \times F_d^*$, that in turn contains $L$.

By the assumption that the tuple  $(F_1,\ldots,F_{d})$ is independent,  $q \Z v_1 +  \ldots+  q \Z v_d$ is a finite index subgroup of $\Z^d$ for every $(v_1,\dots,v_d) \in  F_1^* \times \ldots \times F_d^*$. Since $L$ is an intersection of finitely many finite index subgroups, $L$ is also a finite index subgroup. In particular, this proves that $f$ is $d$-periodic. In particular, any joint co-tile $A \subseteq \Z^d$ for  $(F_1,\ldots,F_{d})$ is $L$-periodic. For any finite-index subgroup $L$ of $\Z^d$ there are finitely many $L$-periodic subsets, so the set of joint co-tiles for $(F_1,\ldots,F_{d})$ is finite.
\end{proof}

\begin{remark}\label{rem:periodic_is_finite}
    It is clear that in any system of tiling equations that admits a finite number of solutions, the stabilizer of solutions is a finite index subgroup. It turns out that the converse is also true: If every solution of a certain system of tiling equations has a finite index stabilizer, it follows that the system admits only finitely many solutions \cite{MR2873724}. See also \cite{MR3950643} for further generalizations of this fact.
\end{remark}

\section{Joint co-tilings in finitely generated abelian groups}\label{sec:countable_abelian}
It is natural to ask which of the results about tilings generalize from $\Z^d$ to more general groups.
An inspection of the proof of \Cref{thm:structure_theorem}  reveals that the statement still holds, and the same proof applies if we replace $\Z^d$ by an arbitrary countable abelian group $\Gamma$, and change the value of  $q$ in \Cref{thm:structure_theorem} \ref{thm_sec_c:structure_theorem} by multiplying it with the product of the orders of all torsion elements in $F-F$.

Recall that elements $g_1,\ldots,g_k$ in a countable abelian group $\Gamma$ are called \emph{independent} if the equation $\sum_{j=1}^k n_j g_j =0$, with $n_1,\ldots,n_k \in \Z$, implies that $n_1=\ldots = n_k=0$.
With this definition, \Cref{thm:periodic_decomposition} extends directly as follows:

\begin{thm}
    Let $\Gamma$ be a countable abelian group. For every $k \in \N$, the indicator function of any joint co-tile for $k$ independent tiles in $\Gamma$ is equal to a sum of $[0,1]$-valued functions whose stabilizer has rank at least $k$, plus an integer constant.
\end{thm}

\begin{remark}\label{rem:fg_abelian_countable_abelian}
 The structure of a general countable abelian group can be quite complicated. In contrast, \emph{finitely generated} abelian groups have a simple structure theorem: Any finitely generated abelian group is isomorphic to $\Z^d \times G$, for some integer $d \ge 0$ and some finite abelian group $G$ (which is a product of cyclic groups of the form $(\Z/m\Z$). However, 
the following simple observation allows one to reduce statements about tilings of countable abelian groups by a finite set to the finitely generated case: Let $\Gamma$ be a countable abelian group and let $F \Subset \Gamma$ with $0\in F$. Let $\Gamma_0$ denote the group generated by the difference set $F-F$. The assumption $0 \in F$ implies that $F \Subset \Gamma_0$. Then for any co-tile $A$ of $F$ we have that $A \cap \Gamma_0$ is a co-tile of $F$ in $\Gamma_0$, and tilings of $\Gamma$ by $F$ decompose into tilings of cosets of $\Gamma_0$ in $\Gamma$. A corresponding statement is true also for a tuple of tiles $(F_1,\ldots,F_k)$ and a joint co-tile.
\end{remark}

\Cref{thm:d_independent_tiles} can be extended to finitely generated abelian groups as follows:

\begin{thm}\label{thm:rank_d_groups_indepednent_co_tile}
Let $\Gamma$ be a finitely generated abelian group of rank $d$.
Any joint co-tile for $d$ independent tiles in $\Gamma$ has a finite orbit.
\end{thm}

\begin{remark}
Using \Cref{rem:fg_abelian_countable_abelian} one  can directly conclude the following property for joint co-tiles of an arbitrary countable group:
Let $\Gamma$ be a countable abelian group and let $(F_1,\ldots, F_d)$ be an independent tuple of tiles in $\Gamma$ so that the group $\Gamma_0$ generated by $\bigcup_{i=1}^d(F_i -F_i)$ has rank $d$. 
Then there are finitely many set $A_1,\ldots,A_m \subset \Gamma_0$ such that the any joint co-tile $A \subseteq \Gamma$ of  $(F_1,\ldots, F_d)$ satisfies $A \cap \Gamma_0 = A_i$ for some $1\le i \le m$.    
\end{remark}

A quick remark about the condition of independence for a tuple of tiles for finitely generated abelian groups with non-trivial torsion:
\begin{remark}
If $\Gamma$ is of the form $\Gamma= \Z^d \times G$ where $G$ is a finite abelian group and $(F_1,\ldots,F_k)$ is an independent tuple of tiles in $\Gamma$, then the only torsion element in each of the sets $F_i$ is $0$. 
For this reason, Newman's theorem (i.e. any tiling of $\mathbb{Z}$ by a finite set is periodic) does not hold in abelian groups $\Gamma$ that are finite extensions of $\Z$.
\end{remark}

\begin{figure}[H]
\centering
        \includegraphics[scale=0.6]{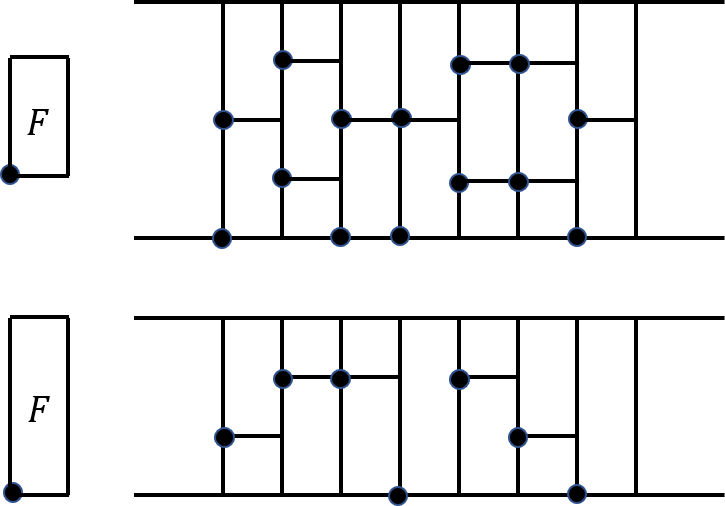}
        \caption{Non-periodic co-tiles of $\Z\times\Z/n\Z$ for $n=4$ and $n=3$ for the tile $F=\{1\}\times G$, where $G$ is a non-trivial subgroup of $\Z/n\Z$.}
\end{figure}\label{pic:entire_fiber}

Indeed, let $\Gamma = \Z \times G$, where $G$ is a finite abelian group. Consider the tile $F = \{1\} \times G \Subset \Gamma$. The co-tiles of $F$ are all the sets $A \subset \Gamma$ of the following form:
\[A= \{(n,g_n):~ n \in \Z \},\]
for some sequence $(g_n)_{n \in \Z}$ of elements in $G$.
In particular, 
 it is no longer true that any co-tile of $F$ must be periodic, unless $G$ is trivial.
Nonetheless, if $G$ is a finite cyclic group of prime order, then the only obstructions to extending Newman's theorem are of this form. See \cite[Lemma 5.1 and Remark 5.2]{Greenfeld_Tao1_2021} for related statements.

\begin{prop}\label{prop:newman_Z_mod_pZ}
    If $\Gamma = \Z \times (\Z/p\Z)$ for some prime number $p$ and $F \Subset \Gamma$ is a finite set, then every co-tile of $F$ is periodic, unless $F$ is of the form $F = \tilde F \times (\Z/p\Z)$ for some finite tile $\tilde F \Subset \Z$, in which case the co-tiles of $F$ are all of the form\begin{equation}\label{eq:A_graph_Z_mod_p}
A = \{ (n,g_n): n \in \tilde A\},~ g_n \in \Z / p\Z,    
\end{equation}
where $\tilde A$ is a co-tile of $\tilde F \Subset \Z$, which by Newman's theorem must be periodic.
\end{prop} 


The proof of the proposition relies on the following generalization of \Cref{thm:structure_theorem}. 
\begin{thm}
\label{thm:structure_countable_abelian} Let $\Gamma$ be a countable abelian group, $F_1,\ldots,F_k \Subset \Gamma$ such that $|F_i| = S$, and $0 \in F_i$ for all $1 \le i \le k$, and let $f:\Gamma\to\Z$ be a bounded function that satisfies $\ind_{F_i} * f = 1$ for all $1 \le i \le k$. 
For every $1 \le i \le k$, let $\tF_i$ denote the intersection of $F_i$ with the torsion subgroup of  $\Gamma$, and let $F_i^* = F_i \setminus \tF_i$.
Then for every $1 \le i \le k$ and every $(v_1,\ldots,v_i) \in F_1^*\times \ldots \times F_i^*$ there exists a function $\phi_{v_1,\ldots,v_i}:\Gamma \to [\min f,\max f]$ with the following properties:
\begin{enumerate}[label=(\alph*)]
\item\label{thm_sec_a:structure_theorem_Gamma}  
For $i < k$ we have
\begin{equation*}
        \ind_{\tF_{i+1}} * \phi_{v_1,\ldots,v_i} = 1 - \sum_{v_{i+1} \in F_{i+1}^*} \phi_{v_1,\ldots,v_i,v_{i+1}}.
\end{equation*}
\item\label{thm_sec_b:structure_theorem_Gamma}  
For every $1\le i \le k$ there is an integer constant $C_i$ such that 
\begin{equation*}
        \ind_{\tF_1} * \ldots * \ind_{\tF_i} * f = (-1)^i\sum_{(v_1,\ldots,v_i) \in F_1^*\times \ldots \times F_i^*} \phi_{v_1,\ldots,v_i} + C_i.
\end{equation*}
\item\label{thm_sec_c:structure_theorem_Gamma} 
Let $q_1$ be the product of all primes less than or equal to $(\max f - \min f)S$, let $q_2$ be the product of all the orders of the torsion elements in the sets $F_i-F_i$, for $1\le i\le k$, and set $q = q_1q_2$. Then
\[
\left( \Z qv_1 + \ldots+ \Z qv_i \right)  \le  \stab(\phi_{v_1,\ldots,v_i}),
\]
\item\label{thm_sec_d:structure_theorem_Gamma} 
$1_{F_j} *\phi_{v_1,\ldots,v_i}=1$ for all $1 \le j \le k$. In particular, $\phi_{v_1,\ldots,v_i}$ has mean $1/S$. 
\end{enumerate}
\end{thm}

 The proof of \Cref{thm:structure_countable_abelian} below is a minor adaptation of the proof of \Cref{thm:structure_theorem}.
Note that in the case where $\Gamma$ is a torsion free abelian group, $\tF_i = \{0\}$. In particular, when $\Gamma=\Z^d$, \Cref{thm:structure_countable_abelian} coincides with \Cref{thm:structure_theorem}.

\begin{proof}
By applying \Cref{lem:dilation} for $F_i$ with $\ell=1$ and $q$ as in \ref{thm_sec_c:structure_theorem_Gamma}  we get $\ind_{rF_i} * f = 1$ for every $r \in q\N+1$. Because $r = 1 \mod q$, we have $r\tF_i = \tF_i$. Since $F_i = \tF_i \uplus F_i^*$ we have
\[
\ind_{\tF_i} * f = 1 - \sum_{v \in F^*_i} \delta_{rv}* f \mbox{ for every } 1 \le i \le k.
\] 
For every $N\in\N$, setting $r = 1+n q$ for $n \in \{1,\ldots,N\}$ and taking average we conclude that for every $1 \le j \le k$ we have 
\begin{equation}\label{eq:avg1_Gamma}
    \ind_{\tF_j} * f = 1 -  \sum_{v_j \in F^*_j}\frac{1}{N} \sum_{n_j = 1}^{N} \delta_{(1+n_j q)v_j} * f.
\end{equation} 
Applying \eqref{eq:avg1_Gamma} with $j=i+1$, convolving both sides by $\delta_{(1+n_1q)v_1+\ldots+(1+n_iq)v_i}$ and taking average over $\frac{1}{N^i}\sum_{n_1,\ldots,n_i = 1}^N$ yields 
\begin{align*}\label{eq:basis_of_induction_Gamma}
\begin{split}
    \ind_{\tF_{i+1}} * \left[ 
    \frac{1}{N^i} \sum_{n_1,\ldots,n_i = 1}^N \delta_{(1+n_1q)v_1+\ldots+(1+n_iq)v_i} * f \right] &= 
    \\
    1 - \sum_{v_{i+1} \in F^*_{i+1}} \frac{1}{N^{i+1}}\sum_{n_1,\ldots,n_i,n_{i+1} = 1}^N & \delta_{(1+n_1q)v_1+ \ldots+(1+n_iq)v_i+(1+n_{i+1}q)v_{i+1}}*f.
    \end{split}
\end{align*}
Defining $\phi_{v_1,\ldots,v_i}^{(N)} = \frac{1}{N^i}\sum_{n_1,\ldots,n_i=1}^N \delta_{(1+n_1q)v_1+\ldots+(1+n_iq)v_i} * f$, as in \eqref{eq:phi_i^N}, we obtain 
\begin{equation}\label{eq:phi_rec_Gamma}
     \ind_{\tF_{i+1}} * \phi^{(N)}_{v_1,\ldots,v_i} = 1 -  \sum_{v_{i+1} \in F^*_{i+1}} \phi^{(N)}_{v_1,\ldots,v_i,v_{i+1}}. 
\end{equation}
Note that \eqref{eq:avg1_Gamma} with $j=1$ becomes $    \ind_{\tF_1} * f = 1 -  \sum_{v_1 \in F^*_1}\phi^{(N)}_{v_1}$. Convolving both sides by $\ind_{\tF_2}$ and using \eqref{eq:phi_rec_Gamma} with $i=1$ gives 
\[
\ind_{\tF_1} * \ind_{\tF_2} * f = 
|\tF_2| -  \sum_{v_1 \in F^*_1} \ind_{\tF_2} * \phi^{(N)}_{v_1} = 
|\tF_2| -  \sum_{v_1 \in F^*_1}\left( 
1 - \sum_{v_2\in F^*_2} \phi^{(N)}_{v_1,v_2}
\right).
\]
By an inductive argument we obtain that for every $N \in \N$ and $1\le i \le k$ there is a constant $C_i\in\Z$, that does not depend on $N$, such that 
\begin{equation}\label{eq:property_b_Gamma}
    \ind_{\tF_1} * \ldots \ind_{\tF_i} * f = C_i + (-1)^i\sum_{(v_1,\ldots,v_i) \in F_1^*\times \ldots \times F_i^*} \phi^{(N)}_{v_1,\ldots,v_i}.
\end{equation} 
Items \ref{thm_sec_a:structure_theorem_Gamma} and \ref{thm_sec_b:structure_theorem_Gamma} follow from \eqref{eq:phi_rec_Gamma} and \eqref{eq:property_b_Gamma} respectively. The rest of the proof is completely identical to the proof of \Cref{thm:structure_theorem} and therefore omitted. 
\end{proof}

\begin{lemma}\label{lem:Z_mod_p_invertible}
Let $p$ be a prime number and let $\emptyset \neq F_0\subsetneqq \Z / p\Z$. Then $\ind_{F_0}$ is an invertible element of the ring $\Q^{\Z / p\Z}$, where multiplication in the ring is convolution.  In other words, there exists $g \in  \Q^{\Z / p\Z}$ such that $g * \ind_{F_0}= \delta_0$.
\end{lemma}
\begin{proof}
Consider the ring $\Q[x]/\langle x^p-1 \rangle$ (with operations of addition and multiplication of polynomials). It is easy to check that this ring is isomorphic as a ring to $\Q^{\Z / p\Z}$, with the operations of pointwise addition and convolution.
The isomorphism is given by identifying an element 
\[ \sum_{i=0}^{p-1} a_i x^i + \langle x^p-1 \rangle  \in \Q[x]/\langle x^p-1 \rangle\]
with the function $f  \in \Q^{\Z / p\Z}$ given by $f(i + p \Z) = a_i$.

Let $F_0\subset \Z / p\Z$ be a non-empty proper subset of $\Z/p\Z$.
Then $\ind_{F_0} \in  \Q^{\Z / p\Z}$ is naturally identified with the  coset of the polynomial  $P(x) = \sum_{ (i+ p\Z) \in F_0} x^i$ in  $\Q[x]/\langle x^p-1 \rangle$. Then the assumption that $F_0$ is a non-empty proper subset of $\Z / p\Z$ implies that the polynomial $P$ is co-prime to the cyclotomic polynomial of order $p$, $\Phi_p = \sum_{i=0}^{p-1} x^i$. Since $P(1)= |F_0| \ne 0$ it follows that $P$ is co-prime to $x-1$. Because $x^p-1= \Phi_p(x)(x-1)$, it follows that $P$ is co-prime to $x^p-1$. Hence there exists polynomials $Q_1,Q_2 \in \Q[x]$ such that 
\[ 1 = Q_1(x) P(x) + Q_2(x)(x^p-1).\]
This means that in the ring $\Q[x]/\langle x^p-1 \rangle$, the coset of $Q_1(x) P(x)$ is the same as the coset of the polynomial $1$. Since the coset of the polynomial $1$ in $\Q[x]/\langle x^p-1 \rangle$ corresponds  to $\delta_0 \in \Q^{\Z / p\Z}$,
this implies that $g * \ind_{F_0} = \delta_0$, where $g \in \Q^{\Z / p\Z}$ is the element corresponding to the coset of $Q_1$.
\end{proof}


\begin{proof}[Proof of \Cref{prop:newman_Z_mod_pZ}]
Let $p$ be a prime number and  $F \Subset \Z \times (\Z / p \Z)$ be a finite set.
Suppose $A \subset \Z \times (\Z / p\Z)$ satisfies $\ind_F * \ind_A = 1$.
Applying \Cref{thm:structure_countable_abelian} with $\Gamma= \Z \times (\Z / p \Z)$ $k=1$, $F_1=F$ and $f= \ind_A$, we conclude that $\ind_{\tF} *\ind_A$ is a sum  functions having infinite stabilizer, hence  $\ind_{\tF} *\ind_A$ is periodic. 

First, assume that there is a set $\tilde F \Subset \Z$ such that $F = \tilde F \times \Z / p\Z$. So $\ind_F =   \ind_{\tilde F \times \{0\}} * \ind_{ \{0\} \times (\Z / p\Z)}$. Thus  $ \ind_{\tilde F \times \{0\}} * \ind_{ \{0\} \times (\Z / p\Z)}* \ind_{A} = 1$. 
This implies that  $\ind_{\{0\} \times (\Z \times p\Z)} * 1_A \le 1$, so for every $n \in  \Z$ there exists at most one element $g_n \in \Z / p\Z$ such that $(n,g_n) \in A$. Hence, in this case, $A$ is of the form \eqref{eq:A_graph_Z_mod_p} for some set $\tilde A \subset \Z$. It follows that $\ind_{\tilde F} * \ind_{\tilde A} =1$, where the convolution here is with respect to the group $\Z$. 

Now suppose that $F$ is not of the above form. This means that there exists $n \in \Z$ such that $F \cap (\{n\} \times \Z / p \Z)$ is a non-empty proper subset of $\{n\} \times (\Z / p \Z)$. By translating $F$ we can assume without loss of generality that $\tF$ is neither empty nor equal to $\{0\} \times (\Z / p\Z)$. Then there exists a non-empty proper subset $F_0 \subset \Z / p\Z$ such that $\tF = \{0\} \times F_0$.
In this case, by \Cref{lem:Z_mod_p_invertible}, there exists $g: \Z / p \Z \to \Q$ such that $g * \ind_{F_0} = \delta_0$, where the convolution is in $(\Z/ p\Z)$.  Let $\tilde g:\Z \times \Z / p\Z \to \Q$ be given by $\tilde g(0,i)= g(i)$ for $i \in \Z / p\Z$ and $g(n,i)=0$ for every $n \in \Z \setminus \{0\}$ and $i \in \Z/p\Z$.
Then 
$\tilde g * \ind_{\tF} = \delta_0$, where this time the convolution is in $\Z \times (\Z/
p\Z)$. Since $\ind_{\tF} * \ind_A$ is periodic, so is $\tilde g * \ind_{\tF} * \ind_A = \ind_A$.

We have thus shown that in the case that $F$ is not of the form $F = \tilde F \times (\Z / p\Z)$ for some $\tilde F \Subset \Z$, every co-tile is periodic. 
\end{proof}

\section{Property $(\star)$ implies $(d-1)$-piecewise periodicity}\label{sec:strong_independency}
In this section, we use property $(\star)$ to deduce \Cref{thm:main1}.
To this end, we will use \Cref{thm:weyl_equidist_zd}, which is a version of Weyl's equidistribution theorem for polynomials in several variables. The relevance of Weyl's equidistribution theorem to our setting comes from \Cref{lem:stabilizers_polynomials} below. We note that similar arguments have appeared earlier in \cite{BPeriodicity2020}, \cite{MR4073398}, and \cite{Greenfeld_Tao1_2021}.

\begin{lemma}\label{lem:stabilizers_polynomials} Suppose  $g,g_1,\ldots,g_m:\Gamma_1 \to \Gamma_2$ are functions, where $\Gamma_1,\Gamma_2$ are abelian groups, such that
$\sum_{i=1}^m g_i =g$. 
Suppose $g$ is a polynomial of degree at most $r \in \N$ with respect to a subgroup $\Gamma_0 \le \Gamma_1$.
For any $1 \le i < j \le m$ define the group  
$L_{i,j} = \stab(g_i)+\stab(g_j)$, and let $L = \bigcap_{ 1\le i < j \le m} L_{i,j} \cap \Gamma_0$. 
Then each $g_i$ is a polynomial of degree at most $\max\{m-1,r\}$ with respect to $L$. In particular, if $\Gamma_0$ and $L_{i,j}$ has finite index in $\Gamma_1$ for every $1 \le i < j \le m$, then $L$ has finite index in $\Gamma_1$, and each $g_i$ is a polynomial with respect to a finite index subgroup of $\Gamma_1$.
    
\end{lemma}

\begin{proof} We prove the claim by induction on $m$. If $m=1$ then $g_1=g$, so the claim holds. For $m>1$, take $v \in L$, then in particular $v \in L_{1,2} \cap \Gamma_0$ and thus $v= v_1 +v_2$ for some $v_1 \in \stab(g_1)$ and $v_2 \in \stab(g_2)$. Note that for every function $f:\Gamma_1 \to \Gamma_2$, the identity $D_v f = D_{v_1} f \circ \sigma_{v_2} + D_{v_2} f$ holds, where $\sigma_{u}:\Gamma_1\to\Gamma_1$ denotes the shift by $u$, $\sigma_{u}(w) = w-u$. Since $D_{v_1}g_1=0$, applying this identity to $g_1= -\sum_{i=2}^m g_i + g$ yields
\[D_v g_1 = D_{v_2} g_1 = -D_{v_2}\left(\sum_{i=2}^m g_i - g\right).\]

Since $D_{v_2} g_2 = 0$ we have
\begin{equation}\label{eq:poly_induction}
    D_{v} g_1 + \sum_{i=3}^m D_{v_2} g_i = D_{v_2} g.
\end{equation} Note that $v_2 \in \Gamma_0$, hence $D_{v_2} g$ is a polynomial of degree at most $r-1$ with respect to $\Gamma_0$. So by the induction hypothesis, each summand on the left-hand side in \eqref{eq:poly_induction} is a polynomial of degree at most $\max\{m-2,r-1\}$ with respect to a subgroup $L'$, defined in a similar way to $L$ using the functions $D_{v}g_1, D_{v_2}g_3,\ldots,D_{v_2}g_m$.
In particular, for every $v\in L$ the function $D_{v} g_1$ is a polynomial of degree at most $\max\{m-2,r-1\}$ with respect to $L'$. 

Now observe that for every $f:\Gamma_1\to\Gamma_2$ and $v\in\Gamma_1$ we have $stab(f)\subseteq stab(D_v f)$, thus $L\le L'$ and for every $v\in L$ we, in particular, have that $D_{v} g_1$ is a polynomial of degree at most $\max\{m-2,r-1\}$ with respect to $L$. 
In a similar way for $2 \le i \le m$ and every $v\in L$, each $D_{v} g_i$ is a polynomial of degree at most $\max\{m-2,r-1\}$ with respect to $L$, which completes the proof.

\end{proof}

\begin{lemma}\label{lem:Weyl_in_action}
    Suppose $g:\Z^d \to [0,1]$ is a function such that:
    \begin{enumerate}
    \item
    $g \mod 1$ is a polynomial with respect to a finite index subgroup of $\Z^d$.
    \item 
    $g$ is a sum of finitely many non-negative $(d-1)$-periodic functions.
    \end{enumerate}
    Then there exists a finite index subgroup $\Gamma \le \Z^d$ such that the restriction of $g$ to 
    each coset of $\Gamma$ is $(d-1)$-periodic.
\end{lemma}
\begin{proof}
Suppose $g=\sum_{i=1}^m g_i$, where $g_i:\Z^d\to[0,1]$ and $\rank(\stab(g_i)) \ge d-1$.
In case that $\rank\left(\bigcap_{i=1}^m \stab(g_i)\right)\ge d-1$, the function $g$ is $(d-1)$-periodic and the assertion follows. Otherwise,  
by summing together some of the $g_i$'s we can assume without loss of generality that $\stab(g_i)+\stab(g_j)$ is a finite index subgroup of $\Z^d$, for every $i \ne j$.
By \Cref{lem:stabilizers_polynomials}, because $g$ modulo $1$ is a polynomial with respect to a finite index subgroup, we conclude that each of the $g_i$'s modulo $1$ are polynomials with respect to a finite index subgroup $\Gamma_0 \le \Z^d$.
Let 
\[\Gamma = \Gamma_0 \cap \bigcap_{i \ne j}\left( \stab(g_i)+\stab(g_j)\right).\]
We will show that $g$ is $(d-1)$-periodic on each coset of $\Gamma$. Since $\Gamma \leq \Gamma_0$, each $g_i$ modulo $1$ is also a polynomials with respect to $\Gamma$. Hence by Weyl's equidistribution theorem (\Cref{thm:weyl_equidist_zd}), every $g_i$ modulo $1$ is either equidistributed or periodic, on each coset of $\Gamma$.

Fix $u \in \Z^d$. Let $g^{(u)}:(u +\Gamma) \to [0,1]$ denote the restriction of $g$ to this coset.
We consider 3 cases:
\begin{enumerate}
\item Suppose there exists $1\le i \le m$ and $v \in (u +\Gamma)$ such that $g_i(v)=1$. Then because
$0 \le g(v) \le 1$ and $g_j(v) \ge 0$, we conclude that $g_j(v)=0$ for all $j \ne i$. But $g_i(v)=1$ implies that $g_i(v+w_1)=1$ for all $w_1 \in \stab(g_i)$ so by the same argument $g_j(v+w_1)=0$ for all $w_1 \in \stab(g_i)$. Thus, $g_j(v+w_1+w_2)=0$ for all $w_1 \in \stab(g_i)$ and $w_2 \in \stab(g_j)$. Since $\Gamma \le \stab(g_i)+\stab(g_j)$, we conclude that $g_j$ is zero on the coset $u+\Gamma$, for all $j \ne i$.
This shows that in this case $g^{(u)}=g_i$ on $u +\Gamma$, and in particular $g^{(u)}$ is $(d-1)$-periodic. So in the remaining cases we can assume that none of the $g_i$'s are equal to one, hence the $g_i$'s obtain values in the interval $[0,1)$. 
\item Suppose there exists $1 \le i\le m$ such that
$g_i$ is equidistributed modulo 1 on $u + \Gamma$.
Let $0 <\epsilon <1$ be smaller than all the non-zero values obtained by the (possibly empty) set of $g_j$ that are periodic modulo 1. 
Because $g_i$ is  equidistributed modulo 1 on $u + \Gamma$, there exists $v \in u +\Gamma$ such that $g_i(v) > 1- \epsilon$. Thus, $g_j(v) < \epsilon$ for all $j \ne i$. 
As in the previous part, using $\Gamma \le \stab(g_i)+\stab(g_j)$, we conclude that $g_j(w) < \epsilon$ for all $j \ne i$ and all $w \in u +\Gamma$. 
This tells us that in particular that $g_j$ is not equidistributed modulo $1$ on $u + \Gamma$. By the choice of $\epsilon$, $g_j(w)=0$ for every periodic $j\ne i$ and every $w \in u +\Gamma$. We conclude also in this case that $g=g_i$ on $u +\Gamma$ and particular $g^{(u)}$ is $(d-1)$-periodic.
\item The remaining case is that for every $i = 1,\ldots,m$, the function $g_i$ takes values in $[0,1)$ and $g_i$ modulo $1$ is periodic on $u+\Gamma$. Since they take values in $[0,1)$, the $g_i$'s themselves are all $d$-periodic. It follows in this case that $g^{(u)}$ is $d$-periodic, as the sum of $d$-periodic functions (and in particular $(d-1)$-periodic).
\end{enumerate}
\end{proof}

\begin{proof}[Proof of \Cref{thm:main1}]
We conveniently assume $d >2$, because the case $d=2$ is covered by \cite{Greenfeld_Tao1_2021}.
Suppose that $A\subset\Z^d$ satisfies $F_i\oplus A=\Z^d$ for all $1\le i\le d-1$, where $(F_1,\ldots,F_{d-1})$ is a tuple of tiles in $\Z^d$ that has property $(\star)$, see \Cref{def:property_star}.  
Let $\phi_{v_1,\ldots,v_{d-1}} : \Z^d \rightarrow [0,1]$ be as in \Cref{thm:structure_theorem}, applied for $k=d-1$ and $f=\ind_A$. Given $(v_1,\ldots,v_{d-2}) \in F_1^*\times\ldots\times F_{d-2}^*$ and a $(d-1)$-dimensional subspace $V < \R^d$ such that $v_1,\ldots,v_{d-2} \in V$, define
\[
\psi_V = \sum_{w_{d-1} \in F^*_{d-1} \cap V}
\phi_{v_1,\ldots,v_{d-2},w_{d-1}}.
\]

Note that by the independence of $(F_1,\ldots,F_{d-1})$, every $(d-1)$-tuple in $F_1^*\times\ldots\times F_{d-1}^*$ spans a $(d-1)$-dimensional subspace. Denote by $H$ the set (counted without multiplicity) of all $(d-1)$-dimensional subspaces of $\R^d$ spanned by $(d-1)$-tuples in $F_1^*\times\ldots\times F_{d-1}^*$, and for $(v_1,\ldots,v_{d-2}) \in F_1^*\times\ldots\times F_{d-2}^*$ let $H(v_1,\ldots,v_{d-2}) \subset H$ be the set of such subspaces of dimension $(d-1)$ that contain $v_1,\ldots,v_{d-2}$.  
Thus, for every fixed tuple $(v_1,\ldots,v_{d-2}) \in F_1^*\times\ldots\times F_{d-2}^*$ we have
\begin{equation}\label{eq:psi_recursion}
\sum_{w_{d-1} \in F_{d-1}^*} \phi_{v_1,\ldots,v_{d-2}, w_{d-1}} = 
\sum_{V \in H(v_1,\ldots,v_{d-2})}\psi_V.
\end{equation}
By property $(\star)$, $\{H(v_1,\ldots,v_{d-2}) ~:~ (v_1,\ldots,v_{d-2}) \in F_1^*\times\ldots\times F_{d-2}^* \}$ is a partition of $H$, therefore 
\begin{equation}\label{eq:phi_to_psi}
    \sum_{(v_1,\ldots,v_{d-1}) \in F_1^*\times\ldots\times F_{d-1}^*} \phi_{v_1,\ldots,v_{d-1}} = 
\sum_{V \in H}\psi_V.
\end{equation}
It follows that the functions $\psi_V$ possess the following three properties: 
\begin{enumerate}[label=(\roman*)]
\item\label{psi_a:psi_are_poly}
\[
1-\phi_{v_1,\ldots,v_{d-2} }= \sum_{V \in H(v_1,\ldots,v_{d-2})} \psi_V.
\]
\item\label{psi_c:psi_are_poly}
 $\stab(\psi_V)$ is  a rank $(d-1)$ subgroup of $V\cap\Z^d$. 
\item\label{psi_b:psi_are_poly}
$\psi_V$ modulo $1$ is a polynomial with respect to a finite index subgroup of $\Z^d$.
\end{enumerate}
Indeed, property \ref{psi_a:psi_are_poly} is a direct consequence of \Cref{thm:structure_theorem} part \ref{thm_sec_a:structure_theorem} with $i=d-1$, combined with \eqref{eq:psi_recursion}.
Property \ref{psi_c:psi_are_poly} follows from \Cref{thm:structure_theorem} part \ref{thm_sec_c:structure_theorem}. Now we show property \ref{psi_b:psi_are_poly}. Setting $\widetilde{\psi_V} = \psi_V \mod 1$, the equation in \Cref{thm:structure_theorem} part \ref{thm_sec_b:structure_theorem} (with $f=\ind_A$ and $i=d-1$), combined with \eqref{eq:phi_to_psi}, yields that $\sum_{V\in H}\widetilde{\psi_V} = 0$. 
Take any two non equal elements $V,V' \in H$, that is $V \ne V'$, by property \ref{psi_c:psi_are_poly} we have that $\stab(\widetilde{\psi_V})$ and $\stab(\widetilde{\psi_{V'}})$ are finite index subgroups of $V \cap \mathbb{Z}^d$ and $V' \cap \mathbb{Z}^d$. Since $V$ and $V'$ are distinct $(d-1)$-dimensional subspaces, it follows that $\stab(\widetilde{\psi_V})+\stab(\widetilde{\psi_{V'}})$ is a finite index subgroup of $\Z^d$. Thus property \ref{psi_b:psi_are_poly} follows from \Cref{lem:stabilizers_polynomials}.

In view of these three properties, \Cref{lem:Weyl_in_action} can be applied to $g=1-\phi_{v_1,\ldots,v_{d-2}}$, for any $(v_1,\ldots,v_{d-2})\in F_1^*\times\ldots\times F_{d-2}^*$. This implies that there is a finite index subgroup $\Gamma_{d-2}\le\Z^d$ such that each $\phi_{v_1,\ldots,v_{d-2}}$ is a polynomial with respect to $\Gamma_{d-2}$, and its restriction to every coset $u+\Gamma_{d-2}$ is $(d-1)$-periodic. 

Next, we iterate the above argument using the recursion formula in part \ref{thm_sec_a:structure_theorem} of \Cref{thm:structure_theorem} combined with \Cref{lem:Weyl_in_action}. In turn, this yields a finite index subgroup $\Gamma_1\le\Z^d$ such that each $\phi_{v_1}$ is a polynomial with respect to $\Gamma_1$, and its restriction to every coset $u+\Gamma_1$ is $(d-1)$-periodic. By part \ref{thm_sec_b:structure_theorem} of \Cref{thm:structure_theorem} with $i=1$ we have that 
\[
1-\ind_A = \sum_{v_1\in F_1^*}\phi_{v_1}.
\]
So applying \Cref{lem:Weyl_in_action} to $g = 1-\ind_A$, we obtain a finite index subgroup $\Gamma \le \Z^d$ such that the restriction of $1 - \ind_A$ to each coset of $\Gamma$ is $(d-1)$-periodic. Hence the restriction of $\ind_A$ to each coset of $\Gamma$ is $(d-1)$-periodic. Thus, if $u_1,\ldots,u_r$ are cosets representatives of $\Gamma$ in $\Z^d$, setting $A_{u_i} = A\cap (u_i+\Gamma)\subset\Z^d$ yields a decomposition $A=A_{u_1}\uplus\ldots\uplus A_{u_r}$ of $A$ into finitely many $(d-1)$-periodic sets, as required.
\end{proof}

\section{From piecewise $(d-1)$-periodicity to $d$-periodicity} \label{sec:weakly_piecewise_to_periodic}
The following lemma extracts an idea that  appears within the proof of \cite[Theorem 5.4]{Greenfeld_Tao1_2021}.
\begin{lemma}\label{lem:piecewise_periodic_strt_1}
Suppose that $f_1,\ldots,f_r,f: \Z^d \to \R$ are bounded functions satisfying $f= \sum_{i=1}^r f_i$. Assume additionally that: 
\begin{enumerate}
    \item $\stab(f_i) + \stab(f_j)$ is a finite index subgroup of $\Z^d$ for all $1 \le i < j\le r$.
    \item   $\stab(f)$ is a finite index subgroup of $\Z^d$.
\end{enumerate}
Then, for each $1 \le j \le r$, the group $\stab(f_j)$ is of finite index in $\Z^d$.
\end{lemma}

\begin{proof}
Let $g_1 = f_1 - f$ and $g_j =  f_j$ for $2 \le j \le r$. Then $g_1 + \ldots +g_r = 0$ and  $\stab(g_i) + \stab(g_j)$ is a finite index subgroup of $\Z^d$ for all $1 \le i < j\le r$. 
Using the fact that $0$ is a polynomial, and applying  \Cref{lem:stabilizers_polynomials}, we get that each $g_i$ is a polynomial with respect to a finite index subgroup of $\Z^d$. But each $g_i$ is bounded. By \Cref{lem:bounded_poly=constant}, a polynomial with respect to a finite index subgroup of $\Z^d$ that is bounded must be constant on cosets of this finite index subgroup.
This implies that for each $1 \le j \le r$ the group $\stab(f_j)$ is of finite index in $\Z^d$.
\end{proof}

\Cref{thm:nice_corollary} is a direct consequence of the above lemma, as shown below.
\begin{proof}[Proof of \Cref{thm:nice_corollary}]
Set $f_j=\ind_{A_j}$, then $\sum_{j=1}^r f_j=1$. Let $L_j\le\Z^d$ be the subgroups of rank at least $d-1$ that stabilizes $A_j$. Note that for every two such subgroups $L_{j_1},L_{j_2} \le \Z^d$, either their intersection has rank $d-1$ or their sum has finite index in $\Z^d$. Assume by contradiction that the intersection of all $L_j$'s is of rank less than $d-1$. By unifying some of the $A_j$'s we can assume without loss of generality that $L_{j_1}+L_{j_2}$ is a finite index subgroup of $\Z^d$ for all $1 \le j_1 < j_2 \le r$. By \Cref{lem:piecewise_periodic_strt_1}, for every $1 \le j \le r$ the group $L_j = \stab(f_j)$ is of finite index in $\Z^d$. Since an intersection of finitely many finite index subgroups of $\Z^d$ has rank $d$, the assumption that $\rank\left(\bigcap_{j=1}^r L_j\right) < d-1$ is false. 
\end{proof}

The following  statement  is a direct generalization of \cite[Theorem $2.3$]{BPeriodicity2020}. The argument used in our proof below is quite similar to the one appearing in {BPeriodicity2020}. 

\begin{lemma}\label{lem:(d-1)_to_d_periodic_functions}
Suppose that $\Sigma \Subset \R$ is a finite set of real numbers, $g_1,\ldots,g_r:\Z^d \to \R$ are finitely supported functions and $f:\Z^d \to \Sigma$ is a $(d-1)$-periodic function such that $g_j * f$ is $d$-periodic for  every $1\le j \le r$. Then there exists a $d$-periodic function $\tilde f:\Z^d \to \Sigma$ such that $g_j*f= g_j * \tilde f$ for every $1\le j \le r$.
\end{lemma}
\begin{proof}
Consider the space 
\[X = \{ \tilde x \in \Sigma^{\Z^d }:\, \forall 1\le j \le r,\, g_j * \tilde x= g_j * f \},\]
and let $\Gamma = \bigcap_{j=1}^r \stab(g_j * f)$. Then $X$ is a $\Gamma$-shift of finite type, and by definition $f \in X$ is a $(d-1)$-periodic point in $X$. 
Apply \Cref{lem:Zd_SFT_periodic_points_from_d_minus_1} to conclude that there exists $\tilde f \in X$ that is $d$-periodic. Any such point $\tilde f$ satisfies the conclusion of the lemma.
\end{proof}

At this stage, we are prepared to present the proof of \Cref{thm:from_weakly_periodic_to_periodic}.

\begin{proof}[Proof of \Cref{thm:from_weakly_periodic_to_periodic}]
Suppose that $A\subseteq\Z^d$ is a piecewise $(d-1)$-periodic joint co-tile for $F_1,\ldots,F_k\Subset \Z^d$. That is, there exists functions $f_1,\ldots,f_r:\Z^d\to \{0,1\}$, each $f_j$ is $(d-1)$-periodic, and $\ind_A = \sum_{j=1}^r f_j$.
Consider first the case that $\rank\left(\bigcap_j \stab(f_j)\right) \ge d-1$. Let
$X= \bigcap_{i=1}^k \{x\in \{0,1\}^{\Z^d} :~ \ind_{F_i}*x=1\}$. The elements of $X$ are precisely the indicator functions of joint co-tiles for $(F_1,\ldots,F_k)$. By the discussion following \Cref{def:SFT}, $X$ is a shift of finite type, and $\ind_A = \sum_{j=1}^r f_j \in X$ is a $(d-1)$ periodic point. Thus by \Cref{lem:Zd_SFT_periodic_points_from_d_minus_1} $X$ contains a $d$-periodic point.

 It remains to consider the case where $\rank\left(\bigcap_j \stab(f_j)\right)< d-1$:
Note that for every two subgroups $L_1,L_2 \le \Z^d$ having rank at least $d-1$, either their intersection has rank at least $d-1$ or $L_1 + L_2$ has finite index in $\Z^d$. So as before, by possibly summing some of the $f_j$'s we can assume without loss of generality that $\stab(f_l)+\stab(f_j)$ is a finite index subgroup of $\Z^d$ for all $1 \le l < j \le r$. 
Now consider the functions $\ind_{F_i}*f_j$. Observe that for every $1\le i\le k$ we have $\sum_{j=1}^r \ind_{F_i}*f_j = 1$, and for every $1\le j\le r$ we have $\stab(f_j)\le\stab(\ind_{F_i}*f_j)$. Thus setting $\Lambda_{i,j}:=\stab(\ind_{F_i}*f_j)$ yields that $\rank\left(\Lambda_{i,j}\right)\ge d-1$ and $\Lambda_{i,l}+\Lambda_{i,j}$ is a finite index subgroup of $\Z^d$, for every $1\le i\le k$ and $1\le l< j\le r$.   
Applying \Cref{lem:piecewise_periodic_strt_1} for each $1\le i\le k$ separately we see that each $\Lambda_{i,j}$ is a finite index subgroup of $\Z^d$. That is, each one of the functions $\ind_{F_i}*f_j$ is $d$-periodic. 
For any fixed $1\le j\le r$, applying \Cref{lem:(d-1)_to_d_periodic_functions} with $g_i = \ind_{F_i}$ and $f=f_j$ and $\Sigma = \{0,1\}$, yields a $d$-periodic function $\tilde f_j:\Z^d\to \{0,1\}$ that satisfies $\ind_{F_i}*\tilde f_j = \ind_{F_i}*f_j$, for all $1\le i\le k$. In particular, the function $f:\Z^d\to \Z$ defined by $f:= \sum_{j=1}^r \tilde f_j$ is bounded, $d$-periodic, and it satisfies 
\[
\forall 1\le i\le k: \qquad \ind_{F_i}*f =  \ind_{F_i}*\left(\sum_{j=1}^r \tilde f_j\right) = \sum_{j=1}^r \ind_{F_i}*\tilde f_j = \sum_{j=1}^r \ind_{F_i}* f_j = 1.
\]
Since $f= \sum_{j=1}^r \tilde f_j$ is a sum of $\{0,1\}$-valued functions and $1_{F_i}* f = 1$, it follows that $f$ itself is $\{0,1\}$-valued, hence $f$ is an indicator of a set $\tilde A$ that satisfies $F_i \oplus \tilde A = \Z^d$, for $1\le i\le k$. Since each $\tilde f_j$ is $d$-periodic, so is $\tilde A$. This completes the proof. 

\end{proof}

\ignore{
\begin{proof}[Proof of \Cref{thm:main1}]
Under the assumptions of \Cref{thm:main1} we may apply \Cref{thm:using_strong_independence} to obtain that $A$ is weakly piecewise $(d-1)$-periodic, and the assertion follows from \ref{thm_sec_b:main} of \Cref{thm:from_weakly_periodic_to_periodic}.
\end{proof}
}

\section{Constructing independent tiles with property $(\star)$ for a periodic co-tile}\label{sec:constructing_brother_tiles}
In this section, we prove \Cref{thm:brother_tiles}. 
We repeatedly rely on the following basic fact:
\begin{lemma}\label{lem:lower_dim_subspaces}
Let $L \le \Z^d$ be a finite index subgroup and let $U_1,\ldots,U_r \subset \R^d$ be affine subspaces of dimension strictly smaller than $d$. Then the set $L \setminus \bigcup_{i=1}^r U_i$ is infinite.
\end{lemma}
\begin{proof}
For $n\in\N$ let $B_n=\{-n,\ldots,n\}^d$. Then there exist $c,c_1,\ldots,c_r>0$ such that $\absolute{B_n\cap L} \ge c n^d$ while $\absolute{B_n\cap U_i}\le c_i n^{\dim U_i}\le c_i n^{d-1}$. In particular, $\absolute{B_n\cap \left(L \setminus \bigcup_{i=1}^r U_i\right)}$ tends to infinity as $n$ tends to infinity.
\end{proof}

\begin{lemma}\label{lem:constructing_brothers} Let $F\Subset\Z^d$, let $A\subseteq\Z^d$ such that $F\oplus A=\Z^d$ and let $L\le\Z^d$ be a subgroup satisfying $A+L=A$. Then for every function $f:F\to L$ the tile set 
\[F_f := \{v + f(v) ~:~ v \in F\}\] satisfies $F_f \oplus A = \Z^d$.
\end{lemma}
\begin{proof} 

Given a function $f:F\to L$, we show that $F_f \oplus A = \Z^d$. 	The condition  $F \oplus A= \Z^d$  can be rewritten as 
$\Z^d= \biguplus_{v \in F}(v+A)$. Since $A +L =A$ and $f(v) \in L$ for every $v \in F$, it follows that $f(v)+A=A$.
Thus, 
\[
\Z^d =  \biguplus_{v \in F}(v + A) =\biguplus_{v \in F}(v + f(v)+ A) 
= \biguplus_{\tilde v \in F_f}(\tilde v + A).
\]
This proves that $F_f \oplus A= \Z^d$.
\end{proof}

Let us introduce some ad-hoc notation:
 \begin{definition}\label{def:V_W}
 Given $d,m \in \mathbb{N}$, an $m$-tuple of vectors $T= (v_1,\ldots,v_m)$ with $v_i \in \Z^d$ for $1\le i \le m$, a function $g:\{1,\ldots,m\} \to \Z^d$, a subset $J \subseteq \{1,\ldots,m\}$ and a subspace $W <\R^d$, let $V_W(T,g,J)$  denote the subspace of $\R^d / W$ obtained by projecting $\spn \{ v_j + g(j):~ j \in J\}$ into $\R^d /W$ via the map $v \mapsto v + W$.
 \end{definition}
 
 \begin{lemma}\label{lem:forcing_hyperplane_property}
		Suppose we are given $d,m \in \N$, an $m$-tuple of vectors $T= (v_1,\ldots,v_m)$ with $v_i \in \Z^d$ for $1\le i \le m$ and finite index subgroup $L \le \Z^d$.
 Then for every finite collection $\mathcal{W}$ of proper subspaces of $\R^d$ there exists a function $g:\{1,\ldots,m\} \to L$ so that for every $J \subseteq \{1,\ldots,m\}$ and every $W \in \mathcal{W}$
 we have  
 \[
 \dim \left( V_W(T,g,J)\right) = \min \{ d-\dim(W), |J|\}.
 \]

\end{lemma}

\begin{proof}
We fix a finite collection $\mathcal{W}$ of proper subspaces of $\R^d$ and prove the claim by induction on $m$. For $m=1$, we only need to choose $g(1) \in L$ such that $v_1+g(1) \not \in  W$ for any $W \in \mathcal{W}$.
This is possible by \Cref{lem:lower_dim_subspaces}. 
Assume by induction that $g(1),\dots, g(m) \in L$ have been defined so that the conclusion holds for every $J \subseteq \{1,\ldots,m\}$ and every $W\in\mathcal{W}$. Using \Cref{lem:lower_dim_subspaces} we can choose $g(m+1) \in L$ that is not contained in any affine hyperplane of the form 
\[U:= -v_{m+1} + \spn\{ v_j + g(j) :~ j \in J\}+W,\] 
where $W \in \mathcal{W}$ and $J$ ranges over subsets of $\{1,\ldots,m\}$ of size at most $d- \dim(W)-1$.
We need to show that for any $J \subseteq \{1,\ldots,m+1\}$ and $W \in \mathcal{W}$ we have $\dim \left( V_W(T,g,J)\right) = \min \{ d -\dim(W), |J|\}$. 
Fix some $J\subseteq\{1,\ldots,m+1\}$ and $W\in\mathcal{W}$. 
The assertion follows from the induction hypothesis in case $(m+1) \not \in J$, so suppose $(m+1) \in J$. 
By the induction hypothesis, $\dim\left( V_W(T,g,J \setminus \{m+1\})\right) = \min \left\{d - \dim(W),|J \setminus \{m+1\}|\right\}$.
If $|J \setminus \{m+1\}| \ge d-\dim(W)$, then $\dim\left( V_W(T,g,J \setminus \{m+1\})\right) =d-\dim(W)$, as required. Otherwise, we have that
\[\dim(V_W(T,g,J \setminus \{m+1\})) = |J \setminus \{m+1\}| =|J|-1.\]
By our choice of $g(m+1)$, we have that 
\[
v_{m+1}+ g(m+1) \not\in \spn\left\{v_j+g(v_j) ~:~ j\in J\setminus\{m+1\}\right\} +W,
\]
so
\[ 
\dim\left( V_W(T,g,J )\right) = \dim (V_W(T,g,J \setminus \{m+1\})) +1 = |J|.
\]
This completes the induction step, hence the proof.
\end{proof}

\begin{proof}[Proof of \Cref{thm:brother_tiles}]
 Suppose $F \oplus A = \Z^d$ where $L\le \Z^d$ is a finite index subgroup satisfying $A+L=A$.
    Write $F^* = \{w_1,\ldots,w_k\}$.
    We apply \Cref{lem:forcing_hyperplane_property} with $m=(d-1) k$, the $m$-tuple $T=(v_1,\ldots,v_m)$, where $v_{kj+i}=w_i$ for $0 \le j \le d-2$, and $1\le i \le k$, and $\mathcal{W} = 
    \left\{\spn\{v\} ~:~ v \in F \right\}$ to obtain a function $g:\{1,\ldots,m\} \to L$ so that for every $J \subseteq \{1,\ldots,m\}$ and every $W \in \mathcal{W}$ we have
    \[
 \dim \left( V_W(T,g,J)\right) = \min \{ d-\dim(W), |J|\}.
 \]
 In other words, the set 
 $\tilde T:= \{ v_k +g(k)~:~ 1\le k \le m\} \subset \Z^d$ has the property that 
 any $d$ elements of $\tilde T$  are independent, and the union of any $d-1$ elements of $\tilde T$ together with any element of $F^*$ are also independent.
    
     For $0 \le j \le d-2$ we set
    \[F_{j+1} = \{0\} \cup \{v_{kj +i} + g(kj+i):~ 1\le i \le k\}= \{0\} \cup \{w_i + g(kj+i): 1 \le i \le k\}.\] 
    
By \Cref{lem:constructing_brothers} we indeed have $F_j \oplus A = \Z^d$ for every $1 \le j \le d-1$. 

To see that $(F_1,\ldots,F_{d-1},F)$ is a $d$-tuple of  independent tiles, note that for any
choice of $(u_1,\ldots,u_{d-1},v) \in F_1^* \times \ldots \times F_{d-1}^*\times F$
, the set $\{u_1,\ldots,u_{d-1},v\}$ is a union of $d-1$ elements of $\tilde T$ together with an element of $F^*$, hence it is an independent set.


Let us check that $\{F_1,\ldots,F_{d-2},F\}$ has the property $(\star)$.  Since property $(\star)$ is vacuous in $d=2$, we assume that $d\ge 3$. Choose two distinct $(d-2)$-tuples
\[(u_1,\ldots,u_{d-2}), (\tilde u_1,\ldots,\tilde u_{d-2}) \in F_1^* \times \ldots \times F_{d-2}^*,\]
and $v,\tilde v \in F^*$. 
 Since the $(d-2)$-tuples are distinct, there exists $1\le i \le d-2$ such that $u_i \ne \tilde u_i$. So $\{u_1,\ldots, u_{d-2}\} \cup \{\tilde u_i\} \cup \{ v\}$ is a subset of $\Z^d$ that consists of $d-1$ elements of $\tilde T$ together with an element of $F^*$, hence it is an independent set.
In particular, 
\[ \spn \left( \{\tilde u_1,\ldots,\tilde u_{d-2}\} \right)  \not\subseteq \spn \left(\{u_1,\ldots,u_{d-2},v\} \right). \]
This shows that 
\[ \spn \left( \{\tilde u_1,\ldots,\tilde u_{d-2}, \tilde v\} \right)  \ne \spn \left(\{u_1,\ldots,u_{d-2},v\} \right),\]
which proves that $(F_1,\ldots,F_{d-2},F)$ has property $(\star)$.
 \end{proof}

\section{Further comments and questions}
\label{sec:further_comments}

\subsection{Integer-valued co-tiles}
Given $F \Subset \Gamma$, we say that a bounded function $f:\Gamma \to \Z$ is an \emph{integer-valued co-tile} for $F$ if $\ind_F * f = 1$. Observe that our proof of \Cref{thm:from_weakly_periodic_to_periodic} holds for integer-valued co-tile as well, thus we have:
\begin{prop} 
Let $k$ and $d$ be positive integers and let $F_1,\ldots, F_k\Subset \Z^d$.
 Suppose that $F_1,\ldots, F_k$ admit an integer-valued joint co-tile $f$ and that $f = \sum_{i=1}^rf_r$, where each $f_i:\Z^d\to \Z$ is bounded and $(d-1)$-periodic. Then $F_1,\ldots, F_k$ admit a $d$-period integer-valued joint co-tile. 
\end{prop} 

It is natural to ask whether the existence of an integer-valued co-tile for $F\Subset\Gamma$ implies the existence of a set $A\subseteq \Gamma$ for which $\ind_F * \ind_A =1$? The simple example below shows that this is not true even for $\Gamma = \Z$ (or for $\Gamma$ a finite cyclic group, here $\Z/18\Z$). Let $F_1 = \{0,1\}$, $F_2=\{0,3,6\}$ and $F= F_1 \oplus F_2= \{0,1,3,4,6,7\}$.
\begin{figure}[H]
\centering
        \includegraphics[scale=0.5]{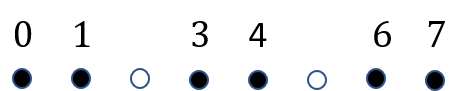}
\end{figure}

We claim that $F$ does not tile $\Z$, but it does admit an integer-valued co-tile. Note that for $A_1 = 2\Z$ and $A_2 = \{0,1,2\}\oplus 9\Z$ we have \[F_1 \oplus A_1 = F_2 \oplus A_2 =\Z.\]
Furthermore, if $\tilde A_1$ is a co-tile for $F_1$ then $\tilde A_1$ must be a translate of $A_1$.
To see that $F$ does not tile $\Z$, suppose by contradiction that $F \oplus A =\Z$ then $F_1 \oplus (F_2 \oplus A) = \Z$, so we must have that $F_2 \oplus A$ is a coset of $2\Z$, but this is clearly impossible since $F_2$ is not contained in a coset of $2\Z$. Now take 
\[
f= \ind_{A_1} -\ind_{A_2}.
\]
Then using $\ind_F = \ind_{F_1} *  \ind_{F_2}$ and $\ind_{F_i} * 1 = |F_i|$ we get:

\[\ind_F * f 
= \ind_{F_2} * (\ind_{F_1} * \ind_{A_1} ) - \ind_{F_1} * (\ind_{F_2} * \ind_{A_2}) 
= |F_2| - |F_1| = 1.\]

\subsection{Conditions for joint tilings for $d$ independent tiles in $\Z^d$}
In view of \Cref{thm:d_independent_tiles}, the classical Wang argument (see \cite{Berger_1966}, \cite{Robinson_1971}) implies that it is algorithmically decidable whether a set of $d$ independent tiles in $\Z^d$ admit a joint co-tile: Indeed, any such tiling must be periodic so we can exhaust the possible periodic co-tiles.
As in \cite{Greenfeld_Tao1_2021}, from an upper bound for the period of a co-tile one can directly deduce an upper bound for the computational complexity of this tiling problem. 
It is of interest to find explicit necessary and sufficient conditions for a $d$-tuple of independent subsets of $\Z^d$ to admit a joint co-tile. In view of \Cref{thm:brother_tiles}, the previous problem is closely related to the more basic question of finding explicit necessary and sufficient conditions for a finite set of $\Z^d$ to tile periodically.

Conversely, one can ask about necessary and sufficient conditions for an infinite subset of $\Z^d$ to be a joint co-tile for $d$-independent tiles. In view of \Cref{thm:d_independent_tiles} and \Cref{thm:brother_tiles}, this is equivalent to the question of finding necessary and sufficient conditions for a periodic subset of $\Z^d$ to be a co-tile for a finite tile.

A complete solution to the above questions involves the factorization of finite abelian groups, namely understanding solutions for $A \oplus B= G$, where $G$ is a finite abelian group. This is a difficult problem even in the cyclic case $G= \Z/M\Z$, which comes up in tilings of $\Z$.

Coven and Meyerowitz \cite{Coven_Meyerowitz_1999} found explicit and efficiently verifiable sufficient conditions for tiling the integers by a finite set. 
It has been conjectured that these conditions are also necessary. This conjecture has been verified in some specific cases recently \cite{Laba_Londner_2022_a, Laba_Londner_2022_b}. 
The necessity of the Coven-Meyerowitz conditions would imply an efficient algorithm for determining if a given finite subset $F \Subset \Z$ can tile $\Z$, see \cite{Kolountzakis_Matolcsi_2009}.


\subsection{Higher level tilings} 
A level $\ell$ co-tile of $\Z^d$ by a finite set $F \Subset \Z^d$ is a set $A \subseteq \Z^d$ such that $\ind_F * \ind_A = \ell$. An obvious generalization of \Cref{prop:mean} implies that if $\ind_F * f = \ell$ then $f$ has mean $\frac{\ell}{|F|}$. A fairly routine modification of the proof of \Cref{thm:structure_theorem} gives the following generalization:


\begin{thm}
\label{thm:level_ell_structure_theorem} Let $\ell_1,\ldots,\ell_k \in \N$, $F_1,\ldots,F_k \Subset \Z^{d}$, with $0 \in F_i$ for all $1 \le i \le k$, and let $f:\Z^d\to\Z$ be a bounded function that satisfies $\ind_{F_i} * f = \ell_i$ for all $1 \le i \le k$. Then for every $1 \le i \le k$ and every $(v_1,\ldots,v_i) \in F_1^*\times \ldots \times F_i^*$ there exists a function $\phi_{v_1,\ldots,v_i}:\Z^d \to [\min f,\max f]$ with the following properties:
\begin{enumerate}[label=(\alph*)]
\item\label{thm_sec_a:structure_theorem_ell}  
For $i < k$ we have
\begin{equation*}
        \phi_{v_1,\ldots,v_i} = \ell_{i+1} - \sum_{v_{i+1} \in F_{i+1}} \phi_{v_1,\ldots,v_i,v_{i+1}}.
\end{equation*}
\item\label{thm_sec_b:structure_theorem_ell}  
\begin{equation*}
        f = (-1)^i\sum_{(v_1,\ldots,v_i) \in F_1^*\times \ldots \times F_i^*} \phi_{v_1,\ldots,v_i} +\sum_{j=1}^{i}(-1)^{j-1} \prod_{t=1}^{j}\ell_t \prod_{s=1}^{j-1}|F_s|.
\end{equation*}
\item\label{thm_sec_c:structure_theorem_ell} 
Let $q$ denote the product of all primes less than or equal to $\max_{1 \le i \le k}\ell_k (\max f - \min f)\max_{1 \le i \le k} |F_i|$, then
\[
\left( \Z qv_1 + \ldots+ \Z qv_i \right)  \le  \stab(\phi_{v_1,\ldots,v_i}),
\]
\item\label{thm_sec_d:structure_theorem_ell} 
$\ind_{F_i} * \phi_{v_1,\ldots,v_i} = \ell_i$ for every $1\le i\le k$. In particular, it has mean $\ell_i/|F_i|$.
\end{enumerate}
\end{thm}

From \Cref{thm:level_ell_structure_theorem} it is easy to deduce that both  \Cref{thm:periodic_decomposition} and \Cref{thm:d_independent_tiles} generalize to level $\ell$ tilings, in the sense that any  ``joint higher-level-co-tile'' for $d$-independent tiles in $\Z^d$ is periodic.

\subsection{Piecewise $1$-periodicity of co-tiles in $\Z^2 \times (\Z/p\Z)$}

By applying the arguments of \Cref{sec:countable_abelian}, the methods of \cite{Greenfeld_Tao1_2021} directly give:
\begin{thm}\label{thm:tiling_Z2_times_prime}
Let $p$ be a prime number, $\Gamma = \Z^2 \times (\Z /p\Z)$ and $F \Subset \Gamma$ be a finite set. 
Then one of the following holds:
\begin{enumerate}
\item Any $A \subset \Gamma$ satisfying $F \oplus A=  \Gamma$ is piecewise $1$-periodic.
\item There exists a finite set $\tilde F \subset \Z^2$ such that $F= \tilde F \times (\Z /p\Z)$.
\end{enumerate}
\end{thm}
Note that according to our definition a subset of $A$ of a finitely generated abelian group $\Gamma$ is $1$-periodic if and only if there an element $\gamma \in \Gamma$ of infinite order such that $\gamma+ A= A$.

In fact, using \Cref{thm:structure_countable_abelian} and the results of \Cref{sec:strong_independency}, we can deduce the following: For any rank $2$ abelian group $\Gamma$ and any  $F \Subset \Gamma$, if $F \oplus A= \Gamma$ then the set $\tF \oplus A$ is piecewise $1$-periodic, whereas in \Cref{sec:countable_abelian}, $\tF$ is the intersection of $F$ with the torsion subgroup of $\Gamma$. Then in the case $\Gamma = \Z^2 \times (\Z/p\Z)$ with $p$ prime, \Cref{lem:Z_mod_p_invertible} implies \Cref{thm:tiling_Z2_times_prime}. 

\begin{cor}\label{cor:tiling_Z2_times_prime}
Let $p$ be a prime number, $\Gamma = \Z^2 \times (\Z /p\Z)$ and $F \Subset \Gamma$ be a finite set. If $F$ tiles $\Gamma$, then $F$ tiles $\Gamma$ periodically.
\end{cor}

Rachel Greenfeld and Terence Tao have informed us in private communication that they also obtained \Cref{cor:tiling_Z2_times_prime}. 

\subsection{A Fourier-analytic and algebraic-geometric approach to the study of joint co-tiles}\label{subsec:fourier}
Fourier analytic methods are a natural approach to translational tiling problems,  see \cite{MR2087242}. Here we outline a specific application related to \cite[Remark 1.8]{Greenfeld_Tao1_2021}. 
Let $g_1,\ldots,g_d:\Z^d \to \C$ be finitely supported functions, by which we mean that $g_i(v)=0$ for all but finitely many $v \in \Z^d$.
Suppose $f:\Z^d \to \C$ is a bounded function that satisfies $g_i * f =1 $ for all $1 \le i \le d$.
Taking distributional Fourier transform on both sides yields
\[ \hat{g_i} \cdot \hat{f} = \delta_0.\]
Thus, the  distributional Fourier transform  of $f$ is supported on $0$ and the 
intersection of the zeros of $\hat {g_i}$.
In particular, if $\hat {g_1},\ldots,\hat{g_d}$ have finitely many common zeros, $f$ must be the  Fourier transform of a multivariate trigonometric polynomial. 
From the multidimensional version of Weyl's equidistribution theorem, it follows that $f$ must be either periodic or equidistributed with respect to a measure whose support is an infinite compact subset of $\C$.
So if we further assume $f$ takes finitely many values (as is the case when $f$ is the indicator function of a joint co-tile), $f$ must be periodic.

The set of common zeros for $d$ polynomials in $d$ variables is ``generically'' a finite set.
Given $v=(n_1,\ldots,n_d) \in \Z_+^d$ let $X^v:= x_1^{n_1}\cdot \ldots \cdot x_d^{n_d}$ denote the corresponding monomial in $d$ variables $x_1,\ldots,x_d$.
Given a finite set $F \Subset \Z_+^d$, let $P_F := \sum_{v \in F} X^v$ denote the corresponding multivariate polynomial.
We conclude that whenever $F_1,\ldots,F_d \Subset \Z_+^d$ are subsets such that the algebraic  variety
\[V(P_{F_1},\ldots,P_{F_d}) := \bigcap_{i=1}^d \left\{ (x_1,\ldots,x_d) \in \C^d ~:~ P_{F_i}(x_1,\ldots,x_d) =0 \right\}\]
has a finite intersection with the $d$-sphere, then any joint co-tile for $F_1,\ldots,F_d$ is periodic.

This raises the question: Is it true that for an independent $d$-tuple $(F_1,\ldots,F_d)$ in $\Z^d$ the algebraic variety $V(P_{F_1},\ldots,P_{F_d})$ is finite?

We note that it can be shown that  $V(P_{F_1},\ldots,P_{F_d})$ is finite if we impose the somewhat stronger condition that  $(F_1-F_1,\ldots,F_d-F_d)$ is an independent $d$-tuple in $\Z^d$. This follows from the equality of the tropical variety with the Bieri-Groves set of the variety (see Theorem  2.2.5 and Corollary 2.2.6 in \cite{EinsiedlerKapranovLind}), combined with \cite[Theorem 2.2.3]{EinsiedlerKapranovLind} and an explicit direct computation.
This connection was kindly explained to us by Ilya Tyomkin.
This argument gives an alternative derivation of  the conclusion of \Cref{thm:d_independent_tiles}, under the slightly stronger assumption that $(F_1-F_1,\ldots,F_d-F_d)$ is an  independent tuple of tiles in $\Z^d$. 

\bibliographystyle{amsplain}

\begin{thebibliography}{10}

\bibitem{MR2873724}
Alexis Ballier, Bruno Durand, and Emmanuel Jeandel, \emph{Structural aspects of tilings}, S{TACS} 2008: 25th {I}nternational {S}ymposium on {T}heoretical {A}spects of {C}omputer {S}cience, LIPIcs. Leibniz Int. Proc. Inform., vol.~1, Schloss Dagstuhl. Leibniz-Zent. Inform., Wadern, 2008, pp.~61--72. 

\bibitem{Beauquier_Nivat_1991}
D.~Beauquier and M.~Nivat, \emph{On translating one polyomino to tile the plane}, Discrete Comput. Geom. \textbf{6} (1991), no.~6, 575--592. 

\bibitem{Berger_1966}
R.~Berger, \emph{The undecidability of the domino problem}, Mem. Amer. Math. Soc. \textbf{66} (1966). 

\bibitem{BPeriodicity2020}
S.~Bhattacharya, \emph{Periodicity and decidability of tilings of {$\mathbb{Z}^2$}}, Amer. J. Math. \textbf{142} (2020), no.~1, 255--266. 

\bibitem{Coven_Meyerowitz_1999}
E.~Coven and A.~Meyerowitz, \emph{Tiling the integers with translates of one finite set}, J. Algebra \textbf{212} (1999), no.~1, 161--174. 

\bibitem{EinsiedlerKapranovLind}
M.~Einsiedler, M.~Kapranov, and D.~Lind, \emph{Non-{A}rchimedean amoebas and tropical varieties}, J. Reine Angew. Math. \textbf{601} (2006), 139--157. 

\bibitem{MR1428636}
Shmuel Friedland, \emph{On the entropy of {$\bold Z^d$} subshifts of finite type}, Linear Algebra Appl. \textbf{252} (1997), 199--220. 

\bibitem{Greenfeld_Tao1_2021}
R.~Greenfeld and T.~Tao, \emph{The structure of translational tilings in {$\mathbb{Z}^d$}}, Discrete Anal. \textbf{16} (2021), 1--28. 

\bibitem{Greenfeld_Tao_2022}
R.~Greenfeld and T.~Tao, \emph{A counterexample to the periodic tiling conjecture}, arXiv:2211.15847 (2022).

\bibitem{Greenfeld_Tao2_2021}
R.~Greenfeld and T.~Tao, \emph{Undecidable translational tilings with only two tiles, or one nonabelian tile}, Discrete and Computational Geometry (2023).

\bibitem{Grunbaum_Shephard_1987}
B.~Gr\"{u}nbaum and G.~C. Shephard, \emph{Tilings and patterns}, W. H. Freeman and Company, New York, 1987. 

\bibitem{MR4073398}
J.~Kari and M.~Szabados, \emph{An algebraic geometric approach to {N}ivat's conjecture}, Inform. and Comput. \textbf{271} (2020), 104481, 25. 

\bibitem{Kenyon_1992}
R.~Kenyon, \emph{Rigidity of planar tilings}, Invent. Math. \textbf{107} (1992), no.~3, 637--651. 

\bibitem{Khetan_2021}
A.~Khetan, \emph{A periodicity result for tilings of $\mathbb{Z}^3$ by clusters of prime-squared cardinality}, arXiv:2109.14179 (2022).

\bibitem{MR2087242}
M.~N. Kolountzakis, \emph{The study of translational tiling with {F}ourier analysis}, Fourier analysis and convexity, Appl. Numer. Harmon. Anal., Birkh\"{a}user Boston, Boston, MA, 2004, pp.~131--187. 

\bibitem{Kolountzakis_Matolcsi_2009}
M.~N. Kolountzakis and M.~Matolcsi, \emph{Algorithms for translational tiling}, J. Math. Music \textbf{3} (2009), no.~2, 85--97. 

\bibitem{Laba_Londner_2022_a}
I.~{\L}aba and I.~Londner, \emph{The {C}oven–{M}eyerowitz tiling conditions for 3 odd prime factors.}, Invent. math., 10.1007/s00222-022-01169-y (2022).

\bibitem{Laba_Londner_2022_b}
I.~{\L}aba and I.~Londner, \emph{Splitting for integer tilings and the coven-meyerowitz tiling conditions}, arXiv:2207.11809v1 (2022).

\bibitem{Lagarias_Wang_1996}
J.~C. Lagarias and Y.~Wang, \emph{Tiling the line with translates of one tile}, Invent. Math. \textbf{124} (1996), 341--365. 

\bibitem{Leibman2002}
A.~Leibman, \emph{Polynomial mappings of groups}, Israel J. Math. \textbf{129} (2002), 29--60.

\bibitem{LindMarcus}
D.~Lind and B.~Marcus, \emph{An introduction to symbolic dynamics and coding}, Cambridge University Press, Cambridge, 1995. 

\bibitem{MR3950643}
Tom Meyerovitch and Ville Salo, \emph{On pointwise periodicity in tilings, cellular automata, and subshifts}, Groups Geom. Dyn. \textbf{13} (2019), no.~2, 549--578. 

\bibitem{Newman_1977}
D.~J. Newman, \emph{Tesselation of integers}, J. Number Theory \textbf{9} (1977), no.~1, 107--111. 

\bibitem{Robinson_1971}
R.~M. Robinson, \emph{Undecidability and nonperiodicity for tilings of the plane}, Invent. Math. \textbf{12} (1971), 177--209. 

\bibitem{Szegedy_1998}
M.~Szegedy, \emph{Algorithms to tile the infinite grid with finite clusters}, Proceedings 39th Annual Symposium on Foundations of Computer Science (Cat. No.98CB36280), 1998, pp.~137--145.

\bibitem{MR2931680}
T.~Tao, \emph{Higher order {F}ourier analysis}, Graduate Studies in Mathematics, vol. 142, American Mathematical Society, Providence, RI, 2012. 

\bibitem{Tijdeman_1995}
R.~Tijdeman, \emph{Decomposition of the integers as a direct sum of two subsets}, Number theory ({P}aris, 1992--1993), London Math. Soc. Lecture Note Ser., vol. 215, Cambridge Univ. Press, Cambridge, 1995, pp.~261--276. 

\bibitem{MR1307966}
Thomas Ward, \emph{Automorphisms of {$\bold Z^d$}-subshifts of finite type}, Indag. Math. (N.S.) \textbf{5} (1994), no.~4, 495--504. 

\bibitem{MR1511862}
Hermann Weyl, \emph{\"{U}ber die {G}leichverteilung von {Z}ahlen mod. {E}ins}, Math. Ann. \textbf{77} (1916), no.~3, 313--352. 

\bibitem{Wijshoff_Leeuwen_1984}
H.~A.~G. Wijshoff and J.~van Leeuwen, \emph{Arbitrary versus periodic storage schemes and tessellations of the plane using one type of polyomino}, Inform. and Control \textbf{62} (1984), no.~1, 1--25. 

\bibitem{Yiftah_2022}
Y.~Yifrach, \emph{A note about weyl equidistribution theorem}, arXiv:2201.07138 (2022).

\end{thebibliography}

\begin{dajauthors}
\begin{authorinfo}[tm]
  Tom Meyerovitch\\
  Ben-Gurion University of the Negev,\\
Department of Mathematics,\\
Beer-Sheva, 8410501, Israel.\\
{\tt mtom@bgu.ac.il}
\end{authorinfo}
\begin{authorinfo}[sa]
  Shrey Sanadhya\\
Ben-Gurion University of the Negev,\\
Department of Mathematics,\\
Beer-Sheva, 8410501, Israel.\\
{\tt sanadhya@post.bgu.ac.il}\\
\hspace{20mm}\\
Einstein Institute of Mathematics,\\ 
The Hebrew University of Jerusalem,\\ 
Edmond J. Safra Campus,\\
Jerusalem, 91904, Israel.\\
{\tt shrey.sanadhya@mail.huji.ac.il}
\end{authorinfo}
\begin{authorinfo}[ys]
  Yaar Solomon\\
  Ben-Gurion University of the Negev,\\
  Department of Mathematics,\\
  Beer-Sheva, 8410501, Israel.\\
  {\tt yaars@bgu.ac.il}
\end{authorinfo}
\end{dajauthors}

\end{document}